\documentclass[12pt]{amsart}
\usepackage{amssymb,latexsym, amsmath, amsxtra}
\usepackage{graphicx}
\newtheorem{theorem}{Theorem}
\newtheorem{lemma}{Lemma}

\newtheorem{proposition}{Proposition}

\begin{document}

\title[Mean square of the product of $L$-functions and Dirichlet polynomials]{The mean square of  the product of a Dirichlet $L$-function and a Dirichlet polynomial}
\author {B. Conrey, H. Iwaniec, and K. Soundararajan}
\address{American Institute of Mathematics, San Jose, CA 95112, USA \\ 
\newline  School of Mathematics, University of Bristol, Bristol BS8 1TW, UK}  
\email{conrey@aimath.org} 
\address{Department of Mathematics, Rutgers University, Piscataway, NJ 08854, USA}
\email{iwaniec@math.rutgers.edu} 
\address{Department of Mathematics, Stanford University, Stanford, CA 94305, USA} 
\email{ksound@stanford.edu} 
 \thanks{The first author is partially supported by the NSF, and by a Programme grant from the EPSRC.   The second author is partially supported by the NSF grant DMS 1406981.  The third author is partially supported by the NSF, and a Simons Investigator grant from the Simons Foundation}  
 
\maketitle

\section{Introduction}  
\noindent  This paper is motivated by work of Balasubramanian, Conrey, and Heath-Brown \cite{BCH-B1985} who established 
that 
\begin{equation}
\label{1.1} 
\int_0^T |A(1/2+it)\zeta(1/2+it)|^2 ~dt
=T\sum_{h,k\le T^{\vartheta}} \frac{\lambda_h\overline{\lambda_k}(h,k)}{{hk}}
\Big( \log \frac{T(h,k)^2}{2\pi h k}+2\gamma-1\Big) +o(T), 
\end{equation}
where  
$$ 
A(s) =\sum_{h\le T^{\vartheta}} \frac{\lambda_h}{h^s}
$$
is a Dirichlet polynomial with length $T^{\vartheta}$ for any fixed $\vartheta <1/2$, and the coefficients $\lambda_h$ satisfy $\lambda_h \ll h^{\epsilon}$.  In \cite{BCH-B1985}, it was conjectured that \eqref{1.1} holds for longer Dirichlet polynomials $A$, again with coefficients $\lambda_h \ll h^{\epsilon}$, and with length $T^{\vartheta}$ with any fixed $\vartheta <1$.  This conjecture implies the Lindel{\" o}f hypothesis, and also that at least $3/5$ of the zeros of the Riemann zeta-function are on the critical line.  
Recent work of Bettin, Chandee, and Radziwill \cite{BCR2014} establishes \eqref{1.1} for Dirichlet polynomials of length $T^{\vartheta}$ provided $\vartheta < \frac12+0.01515$.

One can formulate a more general version of this conjecture, by introducing ``shifts."   Let $T$ be large, and suppose $\alpha$ and 
$\beta$ are two small complex numbers both $\ll 1/\log T$.   Let $h$ and $k$ be natural numbers, and $\psi$ a fixed compactly supported smooth function on ${\Bbb R}^+$.  Then, generalizing \eqref{1.1}, one may conjecture that 
\begin{align} 
\label{1.2}  
\int_0^\infty &\psi\Big(\frac t T\Big) \zeta(\tfrac 12+ it+\alpha)\zeta(\tfrac 12 -it +\beta)(h/k)^{it} ~dt \nonumber \\
&= \int_0^\infty \psi\Big(\frac t T\Big)\Big(\frac{(h,k)^{1+\alpha+\beta}}{h^{1/2+\beta}k^{1/2+\alpha}}
\zeta(1+\alpha+\beta) 
+\Big(\frac t{2\pi}\Big)^{-\alpha-\beta}
\frac{(h,k)^{1-\alpha-\beta}}{h^{1/2-\alpha}k^{1/2-\beta}}\zeta(1-\alpha-\beta)\Big)~dt \nonumber \\ 
&\hskip 1 in +E_{h,k}(T)
\end{align}
where the remainder terms $E_{h,k}(T)$ satisfy 
$$
\sum_{h,k\le T^{\vartheta}}\frac {\lambda_h\overline{\lambda_k}}{\sqrt{hk}}E_{h,k}(T)=o(T)
$$
provided the coefficients $\lambda_h$ are $\ll h^\epsilon$, and $\vartheta <1$.   When $\alpha+\beta=0$, the main term in \eqref{1.2} 
should be interpreted in the sense of a limit.   The introduction of the shifts $\alpha$ and $\beta$ permits the main terms to be 
expressed in a more transparent way, and also allows (via Cauchy's formula) one to deduce variants for derivatives of $\zeta(s)$.

In this paper we consider analogues of the conjectures \eqref{1.1} and \eqref{1.2} in the context of Dirichlet $L$-functions 
to a large modulus $q$.   It has long been known that such $q$-analogues behave similarly to the situation of the zeta-function 
in $t$-aspect (for example, see Selberg \cite{Selberg1946}, Iwaniec and Sarnak \cite{IwaSar1997}, Conrey
\cite{Conrey2007}, and Young \cite{Young2011} among many other papers).    Further, in the context of Dirichlet $L$-functions, we 
can enlarge the family by averaging also over the moduli $q$.   Indeed, it is not difficult to use  the large sieve to prove that 
$$
\sum_{q\le Q}  ~ \sideset{}{^\star}\sum_{\chi \bmod q} |L(1/2,\chi)|^2 |A(1/2,\chi)|^2 \ll Q^{2+\epsilon}, 
$$
where the sum is over primitive characters $\chi$, and 
$$
A(1/2,\chi) = \sum_{n\le Q^{\vartheta}} \lambda_n \frac{\chi(n)}{\sqrt{n}}
$$ 
is a Dirichlet polynomial with $ \lambda_n \ll n^{\epsilon}$ and $\vartheta <1$.   This gives an upper bound of (roughly) the right order of magnitude, while (as in 
conjecture \eqref{1.2}) we seek an asymptotic formula.   In this paper, we use the asymptotic large sieve,  
developed by the authors in \cite{CIS2012, CIS2013, CIS2013a}, to establish a $q$-analogue of the Balasubramanian, Conrey, and Heath-Brown conjecture when averaged also over $q$.

 We now describe more precisely our main result.   Let $\chi \bmod q$ be an even primitive character, and 
 $$
 L(s,\chi) = \sum_{n=1}^{\infty} \chi(n)n^{-s}
 $$
 denote 
the corresponding Dirichlet $L$-function.   The completed $L$-function 
\begin{equation} 
\label{2.1} 
\Lambda(\tfrac 12 + s,\chi) = \Big( \frac{q}{\pi}\Big)^{\frac s2} \Gamma\Big( \frac 14+\frac s2\Big) L(\tfrac 12 +s,\chi), 
\end{equation} 
satisfies the functional equation (with $|\epsilon(\chi)| =1$) 
\begin{equation} 
\label{2.2} 
\Lambda(\tfrac 12+s) = \epsilon(\chi) \Lambda(\tfrac 12-s,\overline{\chi}). 
\end{equation}

Let $W$ denote a fixed $C^{\infty}$ function, compactly supported on $[1,2]$, and let $\alpha$ and $\beta$ be ``shifts."  Our 
goal is to evaluate 
\begin{equation} 
\label{2.3} 
\Delta_{\alpha,\beta}(h,k;Q) :=\sum_q  W\big(\frac qQ\big) \sideset{}{^\flat}\sum_{\chi \bmod q}\Lambda(\tfrac 12+\alpha,\chi)
\Lambda(\tfrac 12 +\beta,\overline{\chi}) \chi(h)\overline{\chi}(k), 
 \end{equation}
 where the $\flat$ indicates that the sum is restricted to even primitive characters.   The restriction to even 
 characters is purely for convenience, so that the $\Gamma$-factors in the functional equation have the same 
 shape, and one can consider in the same way odd primitive characters.  Our interest is in the situation where 
 the shifts $\alpha$ and $\beta$ are small, (precisely, $\alpha$, $\beta \ll 1/\log Q$) but it may be possible to relax this 
 and allow the shifts to be large (see \cite{Bettin2010}).

\begin{theorem} \label{thm1}  Let $Q$ be large, and suppose the shifts $\alpha$ and $\beta$ are $\ll 1/\log Q$.   Then 
\begin{align*}
\Delta_{\alpha,\beta}(h,k) &= \sum_{\substack{ (q,hk)=1}} W\Big( \frac qQ\Big) \Big( \sideset{}{^\flat}\sum_{\chi \bmod q} 1 \Big) \Big( 
\Big(\frac q\pi \Big)^{\frac{\alpha+\beta}{2}} \Gamma\Big( \frac 14+\frac \alpha2\Big) \Gamma\Big(\frac 14+\frac\beta 2 \Big) \frac{(h,k)^{1+\alpha+\beta}}{h^{\frac 12+\beta} k^{\frac 12 +\alpha}} \zeta_q (1+\alpha+\beta) \\
&+ 
\Big( \frac q \pi \Big)^{\frac{-\alpha-\beta}{2} }\Gamma\Big( \frac 14-\frac \alpha 2\Big) \Gamma\Big( \frac 14 -\frac \beta 2 \Big) \frac{(h,k)^{1-\alpha-\beta}}{h^{\frac 12 -\alpha} k^{\frac 12-\beta}} \zeta_q(1-\alpha-\beta)\Big) + {\mathcal E}_{h,k}, 
\end{align*} 
where 
$$ 
\zeta_q(s) = \zeta(s) \prod_{p|q} \Big(1 -\frac{1}{p^s}\Big), 
$$ 
and the remainder terms ${\mathcal E}_{h,k}$ satisfy 
$$
\sum_{h,k\le Q^{\vartheta}} \frac{\lambda_h\overline{\lambda_k}}{\sqrt{hk}} \mathcal{E}_{h,k}= O(Q^{2 -(1-\vartheta)/2 +\epsilon}), 
 $$
uniformly for   arbitrary complex numbers $\lambda_h$ with $\lambda_h \ll h^{\epsilon}$, and $\vartheta <1$.
\end{theorem}

In the statement of the theorem, and throughout the paper, $\epsilon$ will stand for an positive number that 
may be taken arbitrarily small, and its value may change from line to line.   Thus, in the remainder terms in Theorem \ref{thm1} 
we have obtained a power saving whenever $\vartheta$ is fixed below $1$.    
In the situation $\alpha=\beta=0$ (as in the original formulation of the Balasubramanian, Conrey, and Heath-Brown conjecture), a little calculation shows that 
$$ 
\sum_{q} W\Big( \frac{q}{Q} \Big) \sideset{}{^\flat}\sum_{\chi \bmod q} |L(\tfrac 12,\chi)|^2 \chi(h)\overline{\chi}(k) 
$$ 
has for its main term 
$$ 
\sum_{(q,hk)=1} W\Big(\frac qQ\Big) \Big( \sideset{}{^\flat}\sum_{\chi \bmod q} 1 \Big) \frac{\phi(q)}{q} \frac{(h,k)}{\sqrt{hk}} \Big( 
\log \frac{q(h,k)^2}{\pi hk}  +2 \gamma + \frac{\Gamma^{\prime}}{\Gamma}\Big( \frac 14\Big) + 2\sum_{p|q} \frac{\log p}{p-1} \Big), 
$$ 
with the error terms being controlled on average as in Theorem \ref{thm1}.

In the main theorem, we have kept the main term in a natural form, which also suggests that such 
an asymptotic formula holds for each individual $q$ and not just on average over $q$.   But one can readily give an 
asymptotic version of the main terms in Theorem \ref{thm1} where the sum over $q$ has been executed.  
To state this cleanly, we need a little more notation.   For brevity, we shall write 
 \begin{equation} 
 \label{2.4} 
 W_{\alpha,\beta}(x) = x^{1+\frac{\alpha+\beta}{2}} W(x), 
 \end{equation} 
 and write its Mellin transform as 
 \begin{equation} 
 \label{2.5} 
 {\widetilde W}_{\alpha,\beta}(s) = \int_0^{\infty} W_{\alpha,\beta}(x) x^s \frac{dx}{x}.  
 \end{equation} 
 For any complex number $s$, put 
 \begin{equation} 
 \label{2.6} 
 \Phi(q,s) = \prod_{p|q} \Big(1 -\frac{1}{q^s}\Big), 
 \end{equation} 
 and finally, define for complex numbers $s$ and $w$ with $\text{Re}(w) >0$ and Re$(s+w)>1$ 
 \begin{equation} 
 \label{2.7} 
 {\mathcal P}(q;w,s) = \prod_{p\nmid q} \Big(1 -\frac{1}{p^{s+w}} - \frac{2}{p^{1+w}} + \frac{2}{p^{1+s+w}} + \frac{1}{p^{2+2w}} - 
 \frac{1}{p^{2+2w+s}} \Big). 
 \end{equation} 
Then a small calculation allows us to recast the main term in Theorem \ref{thm1} as  
\begin{align} 
\label{1.10} 
 \frac{Q^2}{2}\Phi(hk,1)\bigg( &\widetilde{W}_{\alpha,\beta}(1) \Big( \frac{Q}{\pi} \Big)^{\frac{\alpha+\beta}{2}}  \Gamma\Big(\frac 14 +\frac \alpha 2\Big)
\Gamma\Big(\frac 14 +\frac \beta 2\Big)
  \frac{(h,k)^{1+\alpha+\beta}}{ h^{1/2+\beta}k^{1/2+\alpha}} \nonumber \\
  &\quad\quad\quad\quad \quad \times \zeta(1+\alpha+\beta)  \mathcal{P}(hk;1,1+\alpha+\beta) \nonumber \\
&+\widetilde{W}_{-\beta,-\alpha}(1)
    \Big( \frac{Q}{\pi} \Big)^{\frac{-\alpha-\beta}{2}}  
    \Gamma\Big(\frac 14 -\frac \alpha 2\Big)
\Gamma\Big(\frac 14 -\frac \beta 2\Big)
    \frac{(h,k)^{1-\beta-\alpha}}{ h^{1/2-\alpha}k^{1/2-\beta}} \nonumber\\ 
    &\quad\quad\quad\quad\quad\times
       \zeta(1-\alpha-\beta)  \mathcal{P}(hk;1,1-\beta-\alpha)\bigg),
\end{align}
and in fact it is in this form that we will establish Theorem \ref{thm1}.

As mentioned earlier, Theorem \ref{thm1} establishes a $q$-analogue of the conjecture of Balasubramanian, Conrey, and Heath-Brown, on 
average over $q$.   The extra average over $q$ means that one does not obtain improvements to the convexity bound for Dirichlet $L$-functions from Theorem 1.   However,  one can use Theorem \ref{thm1} to obtain a modest improvement to the result of \cite{CIS2013} 
on critical zeros of Dirichlet $L$-function.  More precisely,  let
$N(T,\chi)$ be the number of zeros of $L(s,\chi)$ whose real parts are positive and imaginary parts at most $T$, and 
let $N_0(T,\chi)$ be the number of such zeros that are on the critical line.
Then for $Q$ sufficiently large, and $(\log Q)^6 \le T \le (\log Q)^A$ (for any $A>6$) we have
\begin{equation*}
 \sum_q  W\big(\frac qQ\big) \sideset{}{^\flat}\sum_{\chi \bmod q}N_0(T,\chi)
> \frac35 \sum_q  W\big(\frac qQ\big) \sideset{}{^\flat}\sum_{\chi \bmod q}N(T,\chi).
\end{equation*}
The earlier result in \cite{CIS2013} gave $57 \%$ instead of $60 \%$ above, and the additional saving comes from the 
flexibility of allowing arbitrary coefficients in the mollifier, which permits the choice of mollifier introduced by Feng \cite{Feng2012}.  
 We also note one difference from our earlier work on the asymptotic large sieve in the context of moments of Dirichlet $L$-function 
 \cite{CIS2012}.  Namely, the earlier work required also a small average in $t$-aspect, whereas in the context of Theorem 1 we 
 are able to treat central values of $L$-functions without this extra $t$-averaging.

\section{Preliminary considerations}


\noindent Throughout, we shall assume that the shifts $\alpha$ and $\beta$ are non-zero, and $\alpha \neq \pm \beta$; the final answers will be uniformly continuous in $\alpha$ and $\beta$, so that the results will hold in the edge cases.   
We begin with an ``approximate functional equation'' for the $L$-functions in our average.  
Define 
\begin{equation} \label{3.1} 
\widetilde{V}_{\alpha,\beta}(s)= \Gamma\Big(\frac{s+1/2 +\alpha}{2}\Big)
\Gamma\Big(\frac{s+1/2 +\beta}{2}\Big) \Big(1 -\Big(\frac{2s}{\alpha+\beta}\Big)^2\Big), 
\end{equation}
which vanishes at $s= \pm \frac 12 (\alpha+\beta)$ simplifying some later calculations.  
For any positive real number $x$, put 
\begin{equation} 
\label{3.2} 
V_{\alpha,\beta} (x)= \frac{1}{2\pi i} \int_{(1)} \widetilde{V}_{\alpha,\beta}(s) x^{-s }~\frac{ds}{s}. 
\end{equation} 
The function $V_{\alpha,\beta}(x)$, which is a variant of the Mellin transform, essentially picks out values of $x$ that are of bounded size.  For large $x$, by moving the line of 
integration in \eqref{3.2} to the right we may establish that $V_{\alpha, \beta}(x) \ll \exp(-\tau x)$ for a suitable positive constant $\tau$.  For 
small $x$, by moving the line of integration in \eqref{3.2} to Re$(s) =-\frac 12+\epsilon$ we can establish that 
$$
V_{\alpha, \beta}(x) = \Gamma(\tfrac 14+\tfrac{\alpha}{2}) \Gamma(\tfrac 14 + \tfrac{\beta}{2}) + O(x^{\frac 12-\epsilon}).
$$
 We shall use these, 
and related facts about $V_{\alpha,\beta}(x)$ without further comment below.   

\begin{lemma} \label{lem1}  With notations as above 
$$ 
 \Lambda(\tfrac 12+\alpha,\chi)\Lambda(\tfrac 12+\beta,\overline{\chi}) 
= S(\alpha,\beta;\chi) + S(-\beta,-\alpha;{\chi})
$$ 
where 
\begin{equation} 
\label{3.3} 
S(\alpha,\beta;\chi)=
\Big(\frac q \pi\Big)^{\frac{\alpha+\beta}{2}}
\sum_{m,n}\frac{\chi(m)\overline{\chi(n)}}{m^{1/2+\alpha}n^{1/2+\beta}}V_{\alpha,\beta} \Big(\frac{\pi mn}{q}\Big). 
\end{equation}
\end{lemma}  
\begin{proof}  We sketch quickly the standard proof.   We start with 
\begin{equation} 
\label{3.35} 
\frac{1}{2\pi i} \int_{(1)} \Lambda(\tfrac 12+ s+\alpha,\chi) \Lambda(\tfrac 12+s+\beta,\overline{\chi}) \Big( 1- \Big(\frac{2s}{\alpha+\beta}\Big)^2\Big) \frac{ds}{s}. 
\end{equation} 
Since we are in the region of absolute convergence, expanding out $L(\tfrac 12+s+\alpha,\chi)L(\tfrac 12+s+\beta,\overline{\chi})$ into 
its Dirichlet series, and integrating term by term, this equals $S(\alpha,\beta;\chi)$.   On the other hand, we may move the 
line of integration to Re$(s) = -1$.  We encounter a pole at $s=0$ which leaves the residue $\Lambda(\tfrac 12+\alpha,\chi)\Lambda(\tfrac 12+\beta,\overline{\chi})$.   Thus the quantity in \eqref{3.35} also equals 
$$ 
\Lambda(\tfrac 12+\alpha,\chi) \Lambda(\tfrac 12+\beta,\overline{\chi}) + \frac{1}{2\pi i} \int_{(-1)} \Lambda(\tfrac 12 + s+\alpha,\chi) \Lambda(\tfrac 12+s+\beta,\overline{\chi}) \Big( 1- \Big(\frac{2s}{\alpha+\beta}\Big)^2\Big) \frac{ds}{s}. 
$$ 
Using the functional equation and replacing $s$ by $-s$, we see that the integral above equals $-S(-\beta, -\alpha;{\chi})$,  
which completes the proof of the lemma.  
\end{proof} 
 
The next lemma follows from the orthogonality relations for characters, adapted to handle even primitive characters.   
 \begin{lemma}
 \label{lem2} 
 If $(mn,q)=1$, then
\begin{equation*}\sideset{}{^\flat}\sum_{\chi\bmod q}\chi(m)\overline{\chi(n)}=
\frac 12\Big( \sum_{ \ d\mid q\atop d\mid (m-n)} \phi(d) \mu\Big(\frac qd\Big)+\sum_{ d\mid q\atop d\mid (m+n)} \phi(d)
\mu\Big(\frac qd\Big)\Big). 
\end{equation*}
 \end{lemma}
 
 It will be convenient to adopt the notation $d | (m\pm n)$ to indicate either $d| (m+n)$ or $d|(m-n)$, and if 
 both possibilities occur then the variable $d$ is counted twice.  Thus, for example, we may combine the two sums 
 appearing in Lemma \ref{lem2} and write $\sum_{d|q, d|(m\pm n)} \phi(d) \mu(q/d)$.

Applying Lemmas \ref{lem1} and \ref{lem2} we obtain 
$$
\Delta_{\alpha,\beta}(h,k)=T_{\alpha,\beta}(h,k)+T_{-\beta,-\alpha}(h,k)
$$
where, writing $q=cd$, and recalling the notation \eqref{2.4}, 
\begin{equation}
\label{3.4} 
T_{\alpha,\beta}(h,k)= \frac Q2 \Big(\frac{Q}{\pi}\Big)^{\frac{\alpha+\beta}{2}} \sum_{m, n} \sum_{ (cd,mnhk)=1\atop d\mid mh\pm nk}
\frac{ \mu(c)\phi(d)}{cd} 
 \frac{1}{m^{\frac 12+\alpha}n^{\frac 12+\beta}}V_{\alpha,\beta}\Big(\frac{\pi mn}{cd}\Big)W_{\alpha,\beta}\Big(\frac{cd}{Q}\Big).
\end{equation}

Our goal is now to evaluate $T_{\alpha,\beta}(h,k)$.  
To this end, we introduce a parameter $C$ (which will eventually be chosen as a small power of $Q$) and split $T_{\alpha,\beta}(h,k)$
as 
$$T_{\alpha,\beta}(h,k)= \mathcal{L}_{\alpha,\beta}(h,k)+\mathcal{D}_{\alpha,\beta}(h,k)+\mathcal{U}_{\alpha,\beta}(h,k),
$$ 
depending on whether $c<C$, and whether $mh = nk$.  Precisely, 
we write 
\begin{equation}
\label{3.5} 
\mathcal{L}_{\alpha,\beta}(h,k)= \frac Q2 \Big(\frac{Q}{\pi}\Big)^{\frac{\alpha+\beta}{2}} 
\sum_{m, n} \sum_{ \substack{(cd,hkmn)=1 \\ {d\mid mh\pm kn} \\ { c>C } }} \frac{\mu(c)\phi(d)}{cd} 
 \frac{1}{m^{\frac 12 +\alpha}n^{\frac 12+\beta}}V_{\alpha,\beta}\Big(\frac{\pi mn}{cd}\Big)W_{\alpha,\beta}\Big(\frac{cd}{Q}\Big),
\end{equation}
\begin{equation}
\label{3.6}
 \mathcal{D}_{\alpha,\beta}(h,k)=\frac Q2 \Big(\frac{Q}{\pi}\Big)^{\frac{\alpha+\beta}{2}}  \sum_{\substack{ m, n \\ mh=nk }} \sum_{\substack{ (cd,hkmn)=1 \\ { c\le C} } }  
  \frac{\mu(c)\phi(d)}{cd}
 \frac{1}{m^{\frac 12+\alpha}n^{\frac 12+\beta}}V_{\alpha,\beta}\Big(\frac{\pi mn}{cd}\Big)W_{\alpha,\beta}\Big(\frac{cd}{Q}\Big),
\end{equation}
and
\begin{equation}
\label{3.7} 
\mathcal{U}_{\alpha,\beta}(h,k)= \frac Q2 \Big(\frac{Q}{\pi}\Big)^{\frac{\alpha+\beta}{2}} \sum_{\substack{ m, n \\ mh\neq nk}} \sum_{ \substack{ (cd,hkmn)=1 \\ {d\mid mh\pm kn} \\ c\le C  }} \frac{ \mu(c)\phi(d)}{cd} 
 \frac{1}{m^{\frac 12+\alpha}n^{\frac 12+\beta}}V_{\alpha,\beta}\Big(\frac{\pi mn}{cd}\Big)W_{\alpha,\beta}\Big(\frac{cd}{Q}\Big).
\end{equation}

In Section 3 we shall give an expression for the diagonal terms ${\mathcal D}_{\alpha,\beta}(h,k)$; when $\vartheta <1/2$ these terms 
account for the entire main term contribution in the main theorem, but for $\vartheta >1/2$ there are additional main terms arising from the 
off diagonal terms.   In Section 4 we treat the terms ${\mathcal L}_{\alpha,\beta}(h,k)$, isolating a potentially large main term.   The most difficult 
part of our analysis involves the terms ${\mathcal U}_{\alpha,\beta}(h,k)$, and this is carried out in Sections 5, 6, and 7.  Briefly, we must 
identify another large term which cancels precisely the corresponding contribution in ${\mathcal L}_{\alpha,\beta}(h,k)$, and then we are left with a 
new main term which when combined with the diagonal term ${\mathcal D}_{-\beta,-\alpha}(h,k)$ leads finally to the main term of the theorem.

\section{The diagonal terms $\mathcal{D}_{\alpha,\beta}(h,k)$}

\noindent This section gives a preliminary treatment of the diagonal terms ${\mathcal D}_{\alpha,\beta}(h,k)$, 
deriving a formula for this term which will be useful in conjunction with similar formulae arising from other main 
terms (yet to be identified).   If we write $h=H(h,k)$ and $k= K (h,k)$, so that $(H,K)=1$, and the relation $mh=nk$ implies that $m=K\ell$ and $n=H\ell$ for some natural number $\ell$.   Therefore 
\begin{equation} 
\label{4.1} 
\mathcal{D}_{\alpha,\beta}(h,k) = \frac Q2  \Big(\frac{Q}{\pi} \Big)^{\frac{\alpha+\beta}{2}} 
\sum_{\ell}  \sum_{ (cd,\ell hk)=1 \atop c\le C} \frac{ \mu(c)\phi(d) }{cd} 
 \frac{1}{\ell^{1+\alpha+\beta}H^{\frac 12+\beta}K^{\frac 12+\alpha}} V_{\alpha, \beta} 
\Big( \frac{\pi \ell^2 HK}{cd} \Big) W_{\alpha,\beta}\Big(\frac{cd}{Q}\Big).  
 \end{equation}

 We can extend the sum in $c$ to include all natural numbers $c$, incurring  an acceptable error term of size (using that $W$ is supported on $[1,2]$ and that $V_{\alpha,\beta}(x)$ is exponentially small for large $x$) 
\begin{equation*}
\ll \frac{Q}{(HK)^{\frac 12}} (\log Q) \sum_{c>C} \frac{1}{c} \sum_{d\le 2Q/c} \frac{\phi(d)}{d} \ll \frac{Q^2(\log Q)}{(HK)^{\frac 12} C}.
\end{equation*} 
If we sum this error term over all $h$ and $k$, we obtain 
\begin{equation} 
\label{4.2} 
\sum_{h, k \le Q^{\vartheta}} \frac{|\lambda_h \lambda_k|}{\sqrt{hk}} \frac{Q^2 \log Q}{\sqrt{HK}C} \ll \frac{Q^{2+\epsilon}}{C}. 
\end{equation}  
Since this is under control (we shall eventually choose $C= Q^{(1-\vartheta)/2}$), we shall omit this error term in our further discussion of ${\mathcal D}_{\alpha,\beta}(h,k)$.

Write 
\begin{equation} 
\label{4.3} 
\phi^*(q) = \sum_{cd = q } \mu(c)\phi(d), 
\end{equation}
which counts the number of primitive characters $\bmod \ q$.   By \eqref{4.1} and \eqref{4.2}, and grouping together terms with $cd=q$, we have (omitting the 
error term as mentioned above)
\begin{align} 
\label{4.4} 
\mathcal{D}_{\alpha,\beta}(h,k)  &=\frac{Q (Q/\pi)^{\frac{\alpha+\beta}{2}}}{2H^{\frac 12+\beta} K^{\frac 12+\alpha}} \sum_{\ell} \sum_{(q,\ell hk)=1}\frac{ \phi^*(q)}{q}  \frac{1}{\ell^{1+\alpha+\beta}} V_{\alpha,\beta} \Big( \frac{\pi \ell^2 HK}{q} \Big) W_{\alpha,\beta}\Big( \frac{q}{Q}\Big).
\end{align} 

Recalling the definition of $V_{\alpha,\beta}(x)$ (see \eqref{3.1} and \eqref{3.2}), we may express the sum over $\ell$ in the 
main term above as (for any $\epsilon >0$, and keeping in mind that $\alpha$ and $\beta$ are $\ll 1/\log Q$) 
$$ 
\frac{1}{2\pi i} \int_{(\epsilon)} {\widetilde V}_{\alpha,\beta}(s) \Big( \frac{q}{\pi HK}\Big)^{s} \zeta_q(1+2s+\alpha+\beta) \frac{ds}{s}, 
$$ 
where 
\begin{equation} 
\label{4.5} 
\zeta_q(s) = \prod_{p\nmid q} \Big(1- \frac 1{p^s}\Big)^{-1} = \zeta(s) \Phi(q,s), \qquad \text{with }  \qquad \Phi(q,s) = \prod_{p|q} \Big(1-\frac 1{p^s}\Big). 
\end{equation} 

To proceed further, and evaluate the sum over $q$ in \eqref{4.4}, the following lemma will be useful. 

\begin{lemma} 
\label{lem4.1} 
Suppose $s$ and $w$ are complex numbers with Re$(s)$ and Re$(w)$ larger than $1$.  Then 
$$ 
\sum_{(q,hk)=1} \frac{\phi^*(q)}{q^{1+w}} \zeta_q(s) = \zeta(w)\zeta(s) \Phi(hk,w) {\mathcal P}(hk;w,s), 
$$ 
where ${\mathcal P}(hk;w,s)$ was defined in \eqref{2.7}.
\end{lemma} 
\begin{proof}  Both sides of the claimed identity can be expressed as products over primes, and so it is enough to check that the Euler factors match up.  For a prime $p | hk$ the Euler factor on both sides is simply $(1-p^{-s})^{-1}$.   If $p\nmid hk$, the Euler factor on the left side is 
$$ 
\Big(1-\frac 1{p^{s}}\Big)^{-1} + \frac{(p-2)}{p^{1+w}} + \sum_{j=2}^{\infty} \Big(1-\frac 1p\Big)^2 p^{-wj} =  
\Big(1-\frac{1}{p^s}\Big)^{-1} + \frac{(p-2)}{p^{1+w}}  + \Big(1-\frac 1p\Big)^{2} \Big(1-\frac 1{p^w}\Big)^{-1} \frac{1}{p^{2w}}. 
$$ 
Multiply this by $(1-p^{-w})(1-p^{-s})$, and then a small calculation shows that the result matches the Euler factor appearing in 
\eqref{2.7}.   
\end{proof}  

Now we return to our analysis of ${\mathcal D}_{\alpha,\beta}(h,k)$, recalling that its main term in \eqref{4.4} equals 
\begin{equation}
\label{4.7} 
\frac{Q(Q/\pi)^{\frac{\alpha+\beta}{2}}}{2H^{\frac 12+\beta} K^{\frac 12+\alpha}} \frac{1}{2\pi i} \int_{(\epsilon)} \frac{\widetilde{V}_{\alpha,\beta}(s) }{(\pi HK)^s }
\sum_{(q,hk)=1} W_{\alpha,\beta}\Big(\frac qQ\Big) \frac{\phi^*(q)}{q} q^s \zeta_q(1+2s+\alpha+\beta) \frac{ds}{s}. 
\end{equation} 

By Mellin inversion, we may write 
\begin{equation} 
\label{4.8} 
W_{\alpha,\beta}(x) = \frac{1}{2\pi i} \int_{(c)} {\widetilde W}_{\alpha,\beta}(w) x^{-w} dw, \qquad \text{where } \qquad 
{\widetilde W}_{\alpha,\beta}(w) = \int_0^{\infty} x^{w} W_{\alpha,\beta}(x) \frac{dx}{x}. 
\end{equation}   
Therefore the sum over $q$ in \eqref{4.7} becomes (for suitably large $c$) 
$$ 
\frac{1}{2\pi i} \int_{(c)} {\widetilde W}_{\alpha,\beta}(w) Q^w \sum_{(q,hk)=1} \frac{\phi^*(q)}{q}q^{-w+s}  \zeta_q(1+2s+\alpha+\beta) dw, 
$$ 
which is, by Lemma \ref{lem4.1}, 
\begin{align*} 
\frac{1}{2\pi i} \int_{(c)} & {\widetilde W}_{\alpha,\beta}(w) Q^{w}  \zeta(w-s) \zeta(1+2s+\alpha+\beta)  \Phi(hk,w-s) 
{\mathcal P}(hk;w-s, 1+ 2s+\alpha+\beta) dw. 
\end{align*}
We move the line of integration to Re$(w) =\epsilon$, passing a pole at $w=1+s$.   Thus the integral above 
equals
$$ 
{\widetilde W}_{\alpha,\beta}(1+s) Q^{1+s}  \zeta(1+2s+\alpha+\beta) \Phi(hk,1) 
{\mathcal P}(hk;1,1+2s+\alpha+\beta) + O(Q^{\epsilon}). 
$$
 
 Using the above analysis in \eqref{4.7} and \eqref{4.4}, we conclude that 
 \begin{align} 
 \label{4.9} 
 {\mathcal D}_{\alpha,\beta}(h,k) &=O\Big( \frac{Q^2 (\log Q)^2}{C\sqrt{HK}}+ \frac{Q^{1+\epsilon}}{\sqrt{HK}}\Big) + 
   \Big( \frac{Q}{\pi}\Big)^{\frac{\alpha+\beta}{2}} \frac{Q^2 \Phi(hk,1)}{2H^{\frac 12+\beta}K^{\frac 12+\alpha}} \frac{1}{2\pi i} \nonumber\\
&\times \int_{(\epsilon)} {\widetilde V}_{\alpha,\beta}(s) {\widetilde W}_{\alpha,\beta}(1+s) \Big(\frac{Q}{\pi HK}\Big)^s 
\zeta(1+2s+\alpha+\beta) {\mathcal P}(hk;1,1+2s+\alpha+\beta) \frac{ds}{s}. 
 \end{align} 
We shall return to this expression later, combining it with other main terms that will arise.  
 

\section{The terms $\mathcal{L}_{\alpha,\beta}(h,k)$}

\noindent To treat $\mathcal{L}_{\alpha,\beta}$, we replace the condition $d\mid (mh\pm nk)$ by a
sum over all even characters modulo $d$. Thus
\begin{align} 
\label{5.1} 
\mathcal{L}_{\alpha,\beta}(h,k)&= Q\Big(\frac{Q}{\pi}\Big)^{\frac{\alpha+\beta}{2}} \sum_{m, n} \sum_{\substack{ {c, d} \\ c> C\\ (cd,hkmn)=1}} \sum_{\substack{ {\psi \bmod d} \\ {\psi \text{ even }}}} 
\frac{\mu(c)}{cd}   \frac{\psi(mh) \overline{\psi}(nk)}{m^{\frac 12+\alpha} n^{\frac 12+\beta}} V_{\alpha, \beta}\Big( 
\frac{\pi mn}{cd} \Big) W_{\alpha,\beta} \Big( \frac{cd}{Q}\Big) \nonumber \\ 
&= {\mathcal L}_{\alpha,\beta}^0(h,k)  + {\mathcal L}_{\alpha,\beta}^{r}(h,k),  
\end{align} 
where $\mathcal{L}_{\alpha,\beta}^0(h,k)$ denotes the contribution of the principal character $\psi=\psi_0$, and $\mathcal{L}_{\alpha,\beta}^r(h,k)$
denotes the contribution from the remaining non-principal characters $\psi$.    It may appear that we are going backwards in this step by introducing characters 
yet again, but the advantage is that the conductor of these characters (essentially $d$) is now smaller (at most $Q/C$) which allows us to use 
the large sieve inequality efficiently.  

We first simplify a little the main term contribution of ${\mathcal L}_{\alpha,\beta}^0(h,k)$, and then 
estimate the contribution of the remainder terms ${\mathcal L}_{\alpha,\beta}^r(h,k)$.  From its definition, we have  
\begin{align} 
\label{5.2} 
{\mathcal L}_{\alpha,\beta}^0(h, k )&= Q\Big(\frac{Q}{\pi}\Big)^{\frac {\alpha+\beta}{2}} \sum_{m,n} \sum_{\substack{ {c,d} \\ {c>C} \\ {(cd,hkmn)=1}} } \frac{\mu(c) }{cd}  
\frac{1}{m^{\frac 12+\alpha} n^{\frac 12+\beta}} V_{\alpha,\beta}\Big( \frac{\pi mn}{cd} \Big) W_{\alpha,\beta}\Big(\frac{cd}{Q}\Big).  
\end{align}
Group terms according to $cd=q$ and since $q$ is of size $Q$ and so in particular $q>1$, 
note that 
$$
\sum_{\substack{ cd = q \\  c>C}} \mu(c) = -\sum_{\substack{ cd=q \\ c\le C}}  \mu(c).
$$  
Thus we may rewrite the sum over $c$ and $d$ in \eqref{5.2} as 
$$ 
-\sum_{\substack{c\le C \\ (c,hkmn)=1}} \sum_{(d,hkmn)=1} \frac{\mu(c)}{cd}  V_{\alpha,\beta} \Big( \frac{\pi mn}{cd}\Big) W_{\alpha,\beta} \Big(\frac{cd}{Q}\Big).
$$ 
Since $\sum_{d\le x,  (d,hkmn)=1} 1 = x\phi(hkmn)/hkmn + O((hkmn)^{\epsilon})$, using partial summation  the above 
equals  
$$ 
- \sum_{\substack{c\le C \\ (c,hkmn)=1}}\frac{ \mu(c) }{c} \frac{\phi(hkmn)}{hkmn} \int_0^{\infty}  V_{\alpha,\beta} \Big( \frac{\pi mn}{cx}\Big) 
W_{\alpha,\beta}\Big(\frac{cx}{Q}\Big) \frac{dx}{x} + O\Big( \frac{C}{Q}(hkmn)^{\epsilon} e^{-\tau mn/Q}\Big), 
$$ 
for some positive constant $\tau$.  
Inserting this in \eqref{5.2}, we conclude that 
   \begin{align} 
  \label{5.3} 
  {\mathcal L}_{\alpha,\beta}^0 (h,k) &= O( Q^{\frac 12+\epsilon} C) \nonumber \\
  & - Q \Big(\frac{Q}{\pi} \Big)^{\frac{\alpha+\beta}{2} } \sum_{m, n} \frac{1}{m^{\frac 12+\alpha}} \frac{1}{n^{\frac 12+\beta}} \frac{\phi(mnhk)}{mnhk} 
  \sum_{\substack{ (c,hkmn)=1 \\ c\le C} } \frac{\mu(c)}{c} 
  \int_0^{\infty}   V_{\alpha,\beta} \Big(\frac{\pi mn}{Qx} \Big) W_{\alpha,\beta}(x) \frac{dx}{x}.
   \end{align}
   We stop our treatment of ${\mathcal L}_{\alpha,\beta}^0(h,k)$ here; it will turn out that this expression 
   cancels another main term arising from our treatment of ${\mathcal U}_{\alpha,\beta}(h,k)$.  
   
 The next lemma gives a satisfactory bound for $\mathcal{L}_{\alpha,\beta}^r(h,k)$, on average 
  over $h$ and $k$.  
  
  \begin{lemma} \label{lem5.1}  Suppose $\lambda_h \ll h^{\epsilon}$ for $1\le h\le Q^{\vartheta}$.  Then 
  $$ 
  \sum_{h, k \le Q^{\vartheta}} \frac{\lambda_h \overline{\lambda_k}}{\sqrt{hk}} {\mathcal L}_{\alpha,\beta}^r (h,k) \ll \frac{Q^{2+\epsilon}}{C} + Q^{1+\vartheta +\epsilon}. 
  $$ 
   \end{lemma}  
  \begin{proof}  Using the Mellin transform \eqref{3.2} in the definition of ${\mathcal L}_{\alpha,\beta}^r(h,k)$, we obtain 
 \begin{align*}
  {\mathcal L}_{\alpha,\beta}^r(h,k)  = Q\Big(\frac{Q}{\pi}\Big)^{\frac{\alpha+\beta}{2}}  \sum_{\substack{ c,d \\ c>C \\ (cd,hk)=1}}& \sum_{\substack{\psi \bmod d \\ \psi \text{ even } 
  \\ \psi \neq \psi_0}} \frac{\mu(c) }{cd}   W_{\alpha,\beta}\Big( \frac{cd}{Q}\Big) \psi(h) \overline{\psi(k)} \\
&\times  \frac{1}{2\pi i} \int_{(\frac 12+\epsilon)} {\widetilde V}_{\alpha,\beta}(s) \Big(\frac{cd}{\pi}\Big)^s \sum_{\substack {m, n \\ (cd, mn)=1} }
\frac{\psi(m)\overline{\psi(n)}}{m^{\frac 12+\alpha+s} n^{\frac 12+\beta+s}} \frac{ds}{s}. 
\end{align*} 

The sum over $m$ and $n$ may be written as 
\begin{equation} 
\label{5.4} 
L(\tfrac 12+\alpha+s,\psi) L(\tfrac 12+\beta+ s,\overline{\psi}) \prod_{p|c} \Big(1-\frac{\psi(p)}{p^{\frac 12+\alpha+s}}  \Big) 
\Big( 1- \frac{\overline{\psi(p)}}{p^{\frac 12+\beta +s}}\Big),  
\end{equation} 
 and since $\psi \neq \psi_0$, we may move the line of integration to Re$(s) =\epsilon$ without 
 encountering any poles.  If $s=\epsilon+it$ then ${\widetilde V}_{\alpha,\beta}(s)/s$ is $\ll e^{-\tau |t|}$ for some positive 
 constant $\tau$, and further the product over $p|c$ in \eqref{5.4} may be bounded by $\ll c^{\epsilon}$.   Therefore, we 
 find (with $s=\epsilon+it$) 
 \begin{align}
 \label{5.5} 
 \sum_{h, k\le Q^{\vartheta}} \frac{\lambda_h \overline{\lambda_k}}{\sqrt{hk}} {\mathcal L}_{\alpha,\beta}^{r}(h,k) 
 &\ll Q^{\epsilon}  \sum_{\substack{c, d \\ c> C} } W\Big( \frac{cd}{Q}\Big) \sum_{\substack{ \psi \bmod d \\ \psi \text{ even} \\ \psi \neq \psi_0}} 
 \Big| \sum_{\substack{ h\le Q^{\vartheta} \\ (h,c)=1} } \frac{\lambda_h \psi(h)}{\sqrt{h}}\Big|^2 \nonumber \\
&\hskip 1 in \times  \int_{-\infty}^{\infty} e^{-\tau |t|} |L(\tfrac 12+\alpha+s,\psi) L(\tfrac 12+\beta+s,\overline{\psi})| dt. 
 \end{align}
 
 Given $c>C$, consider the sum over $d$ above (so $d$ is of size $Q/c$).   We regroup the characters $\psi$ above 
 based on primitive characters; thus suppose $d=ru$ and that $\psi$ is induced by a primitive character $\bmod u$.  Therefore 
 the sum over $d$ above may be bounded by 
\begin{align}
\label{5.6} 
Q^{\epsilon}  \sum_{r \ll Q/c} \sum_{u \ll Q/(cr)}& \sum_{\substack{ \psi \bmod u \\ \psi \text{ even, primitive} }} \Big| \sum_{\substack{h \le Q^{\vartheta} \\ (h,cr) =1} } \frac{\lambda_h \psi(h)}{\sqrt{h} } \Big|^2 
 \nonumber \\
&\times \int_{-\infty}^{\infty} e^{-\tau |t|} \Big( |L(\tfrac 12+\alpha+s,\psi)|^2 +|L(\tfrac 12+\beta+s,\psi)|^2\Big) dt. 
\end{align}
Now applications of the large sieve show that 
$$ 
\sum_{\substack{\psi \bmod u \\ \psi \text{ primitive } } } (|L(\tfrac 12 +\alpha+s, \psi)|^4 + |L(\tfrac 12+ \beta+s,\psi)|^4 ) \ll u^{1+\epsilon} |s|^A,  
$$ 
for some constant $A$, and that 
$$ 
\sum_{u\ll Q/(cr)} \sum_{\substack{\psi \bmod u \\ \psi \text{ primitive } } } \Big| \sum_{\substack{ h\le Q^{\vartheta} \\ (h,cr) =1}} \frac{\lambda_h \psi(h)}{\sqrt{h}} \Big|^4 
\ll Q^{\epsilon} \Big( \frac{Q^2}{c^2 r^2} + Q^{2\vartheta} \Big). 
$$ 
Using these estimates together with Cauchy--Schwarz in \eqref{5.6}, we see that the quantity in \eqref{5.6} is 
$$ 
\ll Q^{\epsilon} \sum_{r \ll Q/c} \frac{Q}{cr} \Big( \frac{Q}{cr} + Q^{\vartheta} \Big) \ll \frac{Q^{2+\epsilon}}{c^2} + \frac{Q^{1+\vartheta + \epsilon}}{c}. 
$$ 
Inserting this estimate into \eqref{5.5} and summing over $c$ with $C\le  c \ll Q$, the lemma follows. 
   \end{proof}

\section{The terms ${\mathcal U}_{\alpha,\beta}(h,k)$: Switching to the complementary modulus}

\noindent We begin our treatment of the terms ${\mathcal U}_{\alpha,\beta}(h,k)$, which forms the hardest part of our analysis.  Recall, 
from \eqref{3.7}, the definition 
\begin{equation} 
\label{6.1} 
{\mathcal U}_{\alpha,\beta}(h,k) = \frac Q2 \Big( \frac{Q}{\pi}\Big)^{\frac{\alpha+\beta}{2}} \sum_{\substack{ m, n \\ mh\neq nk}} \sum_{ \substack{ (cd,hkmn)=1 \\ {d\mid mh\pm kn} \\ c\le C  }} \frac{ \mu(c)\phi(d)}{cd} 
 \frac{1}{m^{\frac 12+\alpha}n^{\frac 12+\beta}}V_{\alpha,\beta}\Big(\frac{\pi mn}{cd}\Big)W_{\alpha,\beta}\Big(\frac{cd}{Q}\Big).
\end{equation} 

Our aim is to replace occurrences of $d$ (which is a divisor of $|mh \pm nk|$) with its 
complementary divisor (essentially $|mh \pm nk|/d$) which will be of smaller size.  The presence of the 
factor $\phi(d)$ and the coprimality condition $(d,hk mn)=1$, makes this a little complicated, entailing extra 
applications of M{\" o}bius inversion.  Treating $m$, $n$ and $c$ as given (with $mh \neq kn$), consider the 
sum over $d$ in \eqref{6.1}.   Write $(mh,nk)=g$, and then the conditions $d|(mh \pm nk)$ and $(d, hk mn) =1$ 
are equivalent to $d| (mh \pm nk)/g$ and $(d,g)=1$.  Therefore, we are interested in 
$$ 
\sum_{\substack{ (d,g)=1 \\ d| (|mh\pm nk|/g)}} \frac{ \phi(d) }{d}  W_{\alpha,\beta}\Big(\frac{cd}{Q}\Big)V_{\alpha,\beta}\Big(\frac{\pi mn}{cd}\Big). 
$$ 
Writing $\phi(d)/d = \sum_{ef= d} \mu(e)/e$, the above becomes 
$$ 
\sum_{(e,g)=1} \frac{ \mu(e)}{e} \sum_{\substack{ (f,g)=1 \\ ef | (|mh\pm nk|/g)}}   W_{\alpha,\beta}\Big( \frac{cef}{Q}\Big) 
V_{\alpha,\beta}\Big(\frac{\pi mn}{cef}\Big). 
$$ 
Lastly, we express the condition $(f,g)=1$ using M{\" o}bius inversion $\sum_{a|g, a|f} \mu(a)$; thus with $ab=f$, the 
sum above is 
\begin{equation} 
\label{6.2} 
\sum_{(e,g)=1} \sum_{a|g} \sum_{\substack{ b \\ abe | (|mh \pm nk|/g)} }  \mu(a) \frac{\mu(e)}{e} W_{\alpha,\beta}\Big( \frac{ceab}{Q}\Big) 
V_{\alpha,\beta}\Big(\frac{\pi mn}{ceab}\Big). 
\end{equation} 

We now introduce the complementary modulus $\ell$ via the relation $|mh \pm nk| = g eab \ell$, thereby converting the 
sum over $b$ into one over $\ell$.   
Thus the quantity in \eqref{6.2} may be recast as 
$$ 
  \sum_{\substack{ (e,g)=1 \\ a|g}} \sum_{\substack{ \ell \\ ae \ell | (|mh\pm nk|/g) } } \mu(a) \frac{\mu(e)}{e}  W_{\alpha,\beta}\Big( \frac{c|mh \pm nk|}{Qg\ell} \Big) 
V_{\alpha,\beta}\Big( \frac{\pi mn g\ell}{c|mh\pm nk|}\Big). 
$$ 
 Inserting this in \eqref{6.1}, we conclude that 
 \begin{align} 
 \label{6.4} 
 {\mathcal U}_{\alpha, \beta}(h,k) & = \frac{Q}{2} \Big(\frac{Q}{\pi}\Big)^{\frac{\alpha+\beta}{2}} \sum_{\substack{ (c,hk)=1 \\ c\le C}} \frac{\mu(c)}{c} 
 \sum_{\substack{ m, n\\ (c,mn)=1 \\ mh\neq nk}} \sum_{ \substack{a | g \\ (e,g)= 1} } \mu (a) \frac{\mu(e)}{e} \frac{1}{m^{\frac 12+\alpha}} \frac{1}{n^{\frac 12+\beta}} \nonumber \\ 
 &\hskip 1 in \sum_{\substack{ \ell \\ ae\ell | (|mh \pm nk|/g)} }W_{\alpha, \beta} \Big( \frac{c|mh\pm nk|}{Qg\ell}\Big) V_{\alpha,\beta} \Big( \frac{\pi mn g\ell}{c|mh \pm nk|}\Big).
 \end{align} 
 
 Since $(mh, nk)=g$, in \eqref{6.4} note that $ae\ell$ must be coprime to $mh/g$ and $nk/g$.  Therefore the congruence $(mh \pm nk)/g  \equiv 0 \bmod{ae\ell}$ may 
 be detected using the characters $\bmod \ {ae\ell}$.  We isolate the contribution of the principal character $\bmod \ {ae\ell}$ as ${\mathcal U}_{\alpha,\beta}^0 (h,k)$ 
 and denote by ${\mathcal U}_{\alpha,\beta}^{r}(h,k)$ the contribution of the remaining non-principal characters.   Thus 
 \begin{align} 
 \label{6.5} 
 {\mathcal U}_{\alpha, \beta}^0(h,k) & = \frac{Q}{2} \Big(\frac{Q}{\pi}\Big)^{\frac{\alpha+\beta}{2}} \sum_{\substack{ (c,hk)=1 \\ c\le C}} \frac{\mu(c)}{c} 
 \sum_{\substack{ m, n\\ (c,mn)=1\\ mh\neq nk}} \sum_{ \substack{a | g \\ (e,g)= 1} } \mu (a) \frac{\mu(e)}{e}\frac{1}{m^{\frac 12+\alpha}} \frac{1}{n^{\frac 12+\beta}}  \nonumber \\ 
 &\hskip 1 in \sum_{\substack{ \ell \\ (ae \ell, mnhk/g^2)=1} } \frac{1}{\phi(ae\ell)} W_{\alpha, \beta} \Big( \frac{c|mh\pm nk|}{Qg\ell}\Big) V_{\alpha,\beta} \Big( \frac{\pi mn g\ell}{c|mh \pm nk|}\Big), 
 \end{align} 
 and 
 \begin{align} 
 \label{6.6} 
 {\mathcal U}_{\alpha, \beta}^r(h,k) & = \frac{Q}{2} \Big(\frac{Q}{\pi}\Big)^{\frac{\alpha+\beta}{2}} \sum_{\substack{ (c,hk)=1 \\ c\le C}} \frac{\mu(c)}{c} 
 \sum_{\substack{ m, n\\ (c,mn)=1 \\ mh \neq nk}} \sum_{ \substack{a | g \\ (e,g)= 1} } \mu (a) \frac{\mu(e)}{e}\frac{1}{m^{\frac 12+\alpha}} \frac{1}{n^{\frac 12+\beta}}  \nonumber \\ 
 &\hskip 0.4 in \sum_{\substack{ \ell } } \frac{1}{\phi(ae\ell)}  \sum_{\substack{ \psi \bmod {ae\ell} \\ \psi\neq \psi_0}} \psi(mh/g) \overline{\psi}(\mp nk/g) W_{\alpha, \beta} \Big( \frac{c|mh\pm nk|}{Qg\ell}\Big) V_{\alpha,\beta} \Big( \frac{\pi mn g\ell}{c|mh \pm nk|}\Big).
 \end{align} 
 The contribution of the non-principal characters will be estimated in the next section, and then we continue the analysis of the main terms arising from ${\mathcal U}_{\alpha,\beta}^0(h,k)$ in Section 7.   

\section{The error terms $\mathcal{U}_{\alpha,\beta}^r(h,k)$}

\noindent Our goal in this section is to establish the estimate 
$$ 
{\mathcal U}_{\alpha,\beta}^E = \sum_{h, k \le Q^{\vartheta}} \frac{\lambda_h \overline{\lambda_k}}{\sqrt{hk}} {\mathcal U}_{\alpha,\beta}^r (h,k) \ll Q^{1+\vartheta +\epsilon}C. 
$$ 
We begin by estimating the contribution of terms arising from \eqref{6.6} with $a\ell \ge Q^{\vartheta +\epsilon} C$ or $ae\ell \ge Q^{12}$.  
Consider first the contribution of terms where $mh/g \equiv \pm nk/g 
\bmod{ae\ell}$.    Since $mh \neq nk$, these terms satisfy $\max(mh/g, nk/g) \gg ae\ell$ and, since $W_{\alpha,\beta}$ is supported on $[1,2]$, we also have that 
$c|mh \pm nk| \asymp Qg \ell$, so that (since $a$ divides $g$) $\max(m, n) \gg Qa\ell/(CQ^{\vartheta})$.   It follows that  for 
some constant $\tau >\tau_1 > 0$ 
$$ 
\Big| W_{\alpha,\beta}\Big( \frac{c|mh \pm nk|}{Qg\ell} \Big) V_{\alpha,\beta}\Big(\frac{\pi mng\ell}{c|mh\pm nk|} \Big) \Big| \ll \exp\Big(- \tau \frac{mn}{Q} \Big) \ll 
\frac{Q}{mn} \exp\Big( - \max\Big( \frac{\tau_1 ae\ell}{Q^2}, \frac{\tau_1 a\ell}{CQ^{\vartheta}}\Big) \Big). 
$$
This is exponentially small when either $a\ell \ge Q^{\vartheta+\epsilon} C$ or $ae\ell \ge Q^{12}$, and 
therefore the contribution of these terms is 
$$ 
Q^{1+\epsilon} \sum_{h, k\le Q^{\vartheta}} \frac{1}{\sqrt{hk}} \sum_{c\le C} \frac{1}{c} \sum_{\substack{ a, e, \ell \\ ae\ell \ge Q^{12}}} \frac{1}{e} 
\sum_{\substack{m, n\\ mh \neq nk \\  a|g \\ ae\ell | (mh/g \pm nk/g)}} \frac{Q}{(mn)^{\frac{5}{4}}} Q^{-10} \ll Q^{-1}.  
$$
Now consider the contribution of terms with $mh/g \not\equiv \pm nk/g \bmod{ae\ell}$.  These terms contribute 
$$ 
\ll Q^{1+\epsilon}  \sum_{h, k\le Q^{\vartheta}} \frac{1}{\sqrt{hk}} \sum_{c\le C} \frac{1}{c} \sum_{\substack{ a, e, \ell \\ ae\ell \ge Q^{12} \\ \text{or } a\ell \ge CQ^{\vartheta+
\epsilon} }} \frac{1}{e\phi(ae\ell)}  \sum_{m, n} \frac{1}{(mn)^{\frac 12-\epsilon}} \exp\Big( -\frac{\tau mn}{2Q} -\frac{\tau a\ell}{2CQ^{\vartheta}}\Big) \ll Q^{-1}. 
$$

Having discarded the contribution from terms with $ae\ell \ge Q^{12}$ or $a\ell \ge CQ^{\vartheta +\epsilon}$, we conclude that 
\begin{align} 
\label{7.01} 
{\mathcal U}_{\alpha,\beta}^{E} &\ll Q^{1+\epsilon} \sum_{c\le C} \frac 1c \sum_{\substack{a, e, \ell \\ ae\ell \le Q^{12}  \\ a\ell \le CQ^{\vartheta +\epsilon}}} \frac{1}{ae^2 \ell} 
\sum_{\substack{ \psi \bmod {ae\ell} \\ \psi\neq\psi_0}} \Big( |{\mathcal U}^+(c,a,e,\ell; \psi) | + |{\mathcal U}_{\alpha,\beta}^{-}(c,a,e,\ell;\psi)| \Big), 
\end{align} 
where 
\begin{align} 
\label{7.02} 
{\mathcal U}_{\alpha,\beta}^{-}& (c,a,e,\ell;\psi) \nonumber\\ 
=& \sum_{\substack{ h, k\\ (hk,c)=1}} \frac{\lambda_h \overline{\lambda_k}}{\sqrt{hk}} \sum_{\substack{m, n \\ (mn,c)=1 \\ a|g \\ (g,e)=1}} 
\frac{\psi(mh/g)\overline{\psi}(nk/g)}{m^{\frac 12+\alpha} n^{\frac 12+\beta}} W_{\alpha,\beta}\Big( \frac{c|mh - nk|}{Qg\ell}\Big) V_{\alpha,\beta}\Big(\frac{\pi mn g\ell}{c|mh - nk|} \Big),  
\end{align}
and ${\mathcal U}_{\alpha,\beta}^{+}$ is defined similarly.  
 To estimate the sums in \eqref{7.01}, we divide the sums over $a$, $e$, and $\ell$ into dyadic blocks $A\le a <2A$, $E\le e <2E$, and $L \le \ell < 2L$.  We may assume that 
 $AL \ll CQ^{\vartheta+\epsilon}$ and that $AEL \ll Q^{12}$.   After restricting these 
 variables to dyadic blocks, we now focus on the contribution of the sum in \eqref{7.02}.   The ${\mathcal U}_{\alpha,\beta}^{+}$ term can be handled similarly, and in fact is simpler to treat (because while $mh-nk$ can become unusually small in size, $mh+nk$ cannot).  The variable $g$ above arises as the gcd of $mh$ and $nk$, and our first task is to 
 separate this variable from $m$, $h$, $n$, and $k$ so as to make the sums over those variables independent of each other.

 Write $g_1= (h,k)$, $g_2=(m,n)$, $g_3=(m/g_2,k/g_1)$, and $g_4 = (n/g_2, h/g_1)$.  Then $g=(mh,nk) = g_1g_2g_3g_4$.  Further, write 
 $h=g_1 g_4 H$, $k=g_1 g_3 K$, $m=g_2g_3M$ and $n= g_2g_4N$.  There are now a number of coprimality conditions $(g_3, g_4)=1$, $(H,g_3)=1$, $(K,g_4)=1$, 
 $(H,K)=1$, $(M,g_4)=1$, $(N,g_3)=1$, and $(M,N)=1$, and further we must have $a|(g_1g_2g_3g_4)$, $(g_1g_2g_3g_4, ec)=1$, and $(MNHK,c)=1$.  
 Thus we may rewrite ${\mathcal U}^-(c,a,e,\ell;\psi)$ as 
 \begin{equation} 
 \label{7.03} 
 \sum_{\substack{g_1, g_2, g_3, g_4 \\ M, N, H, K }}^{*}  \frac{\lambda_{g_1g_4H} \overline{\lambda_{g_1g_3K}}}{g_1 \sqrt{g_3g_4 HK}} \frac{\psi(MH) \overline{\psi}(NK)}{(g_2g_3M)^{\frac 12+\alpha} (g_2 g_4 N)^{\frac 12+\beta}} W_{\alpha,\beta}\Big( \frac{c|MH- NK|}{Q\ell}\Big) V_{\alpha,\beta}\Big( \frac{\pi g_2^2 g_3 g_4 MN \ell}{c|MH - NK|}\Big). 
 \end{equation} 
 Here the $*$ indicates the various coprimality and divisibility conditions mentioned above.

 We now introduce a smooth function $\Psi(x)$ (defined on $[0,\infty)$) with $\Psi(x) = 1$ for $x\le 1$ and $\Psi(x)=0$ for $x\ge 2$.  Next define, for positive 
 real numbers $x$, $y$, and $u$ 
 $$ 
 {\mathcal V}_{\alpha,\beta}(x,y;u) = W_{\alpha,\beta}(|x-y|) \Psi\Big( \frac{xL}{C^2 Q^{\vartheta}}\Big) \Psi\Big( \frac{yL}{C^2 Q^{\vartheta}}\Big) V_{\alpha,\beta}\Big(\frac{u}{|x-y|}\Big). 
 $$ 
 The support of $W$ guarantees that $1\le |x-y| \le 2$, and therefore ${\mathcal V}_{\alpha,\beta}(x,y;u)$ decays exponentially in $u$.  
 In \eqref{7.03} we replace the weights 
 $$ 
 W_{\alpha,\beta}\Big( \frac{c|MH- NK|}{Q\ell}\Big) V_{\alpha,\beta}\Big( \frac{\pi g_2^2 g_3 g_4 MN \ell}{c|MH - NK|}\Big) \text{   with   } {\mathcal V}_{\alpha,\beta}\Big( 
 \frac{cMH}{Q\ell}, \frac{cNK}{Q\ell} ; \frac{\pi g_2^2g_3g_4MN}{Q}\Big).
 $$ 
 The two expressions above are different only if either $M$ or $N$ is at least $CQ$, and in this case the difference is exponentially small because of the 
 rapid decay of $V_{\alpha,\beta}$; therefore we incur a negligible error term in ${\mathcal U}^E_{\alpha,\beta}$ in making this replacement.  
 To proceed further, we introduce a three variable Mellin transform of ${\mathcal V}_{\alpha,\beta}$, and discuss its analytic properties. 
 
 \begin{lemma} \label{lem6.1}  Let $s_1$, $s_2$ and $z$ denote complex numbers with positive real part.  Define 
 $$ 
 \widetilde{\mathcal V}_{\alpha,\beta}(s_1,s_2;z) = \int_0^{\infty}\int_{0}^{\infty} \int_{0}^{\infty} {\mathcal V}_{\alpha,\beta}(x,y;u) x^{s_1} y^{s_2} u^z \frac{dx}{x} \frac{dy}{y} 
 \frac{du}{u}. 
 $$ 
 Then for any integers $k_1$, $k_2 \ge 1$ we have 
 $$ 
 \widetilde{\mathcal V}_{\alpha,\beta}(s_1,s_2;z)  \ll \frac{ (1+C^2Q^{\vartheta}/L)^{k_1-1+\text{Re}(s_1+s_2)}}{\max (|s_1|, |s_2|)^{k_1} }\frac{1}{|z|^{k_2}}, 
 $$ 
 and moreover the following Mellin inversion formula holds: for positive numbers $c_1$, $c_2$, $d$  
 $$ 
 {\mathcal V}_{\alpha,\beta}(x,y;u) = \frac{1}{(2\pi i)^3} \int_{\substack{ \text{Re}(s_1) = c_1 \\ \text{Re}(s_2) = c_2 \\ \text{Re}(z) =d  }} 
 \widetilde{\mathcal V}_{\alpha,\beta}(s_1, s_2; z) x^{-s_1} y^{-s_2} u^{-z} ds_1 ds_2 dz. 
 $$ 
 \end{lemma} 
\begin{proof}  Despite the appearance of $|x-y|$, since $W_{\alpha,\beta}$ is supported away from $0$, the function ${\mathcal V}_{\alpha, \beta}(x,y;u)$ is 
smooth in all three variables.  The estimate on the Mellin transform follows by integrating by parts $k_2$ times in the $u$ variable, and $k_1$ times 
in either $x$ or $y$ corresponding to whether $s_1$ or $s_2$ has larger magnitude, and keeping in mind that ${\mathcal V}_{\alpha,\beta}(x,y;u)=0$ unless $|x- y| \asymp 1$.
\end{proof} 
 
 We use the Mellin inversion formula above in \eqref{7.03}.   Initially we start with $\text{Re}(s_1)= \text{Re}(s_2) = 1/2+\epsilon$ and Re$(z) = \epsilon$, where 
 the sums over $M$ and $N$ are absolutely convergent.   After relating those sums to $L$-functions attached to non-principal characters, we may move the 
 lines of integration to Re$(s_1) = \text{Re}(s_2)= \epsilon$.   Thus we can bound the quantity in \eqref{7.03} by 
 \begin{align} 
 \label{7.04} 
Q^{\epsilon}  \sum_{g_1,g_2,g_3,g_4}^* \frac{1}{g_1g_2g_3 g_4} &\int_{\substack{\text{Re}(s_1) =\epsilon \\ \text{Re}(s_2) = \epsilon \\ 
\text{Re}(z) = \epsilon} } |{\mathcal V}_{\alpha,\beta}(s_1,s_2;z)| |L(\tfrac 12+ \alpha+s_1+z, \psi) L(\tfrac12 + \beta+s_2 + z, \overline{\psi})| \nonumber \\
&\times \Big|\sum_{H, K}^{*} \frac{\lambda_{g_1 g_4 H} \overline{\lambda}_{g_1 g_3 K}}{H^{\frac 12 +s_1} K^{\frac 12 +s_2}} \psi(H) \overline{\psi}(K) \Big| |ds_1 ds_2 dz |. 
 \end{align} 
 Here the factor $Q^{\epsilon}$ arises from estimating trivially the ratio of the actual sum over $M$ and $N$ (which has 
 coprimality restrictions) and its dominant part (which is the product of $L$-functions given above).   
 Further, in \eqref{7.04}, the $*$ over the first sum over $g_1$, $g_2$, $g_3$, $g_4$ keeps track of the coprimality conditions of these variables together with the requirement that 
 $a|g_1g_2g_3g_4$, and the $*$ over the sum over $H$ and $K$ indicates $(H,K)=1$ as well as $(H, g_3c) =(K,g_4c)=1$.   Now we use M{\" o}bius inversion $\sum_{k|(H,K)} 
 \mu(k) = 1$ if and only if $(H,K)=1$ in order to separate the variables $H$ and $K$ in the sum in \eqref{7.04}.  Writing 
 \begin{equation}
 \label{7.05} 
 A(s,\psi;u,v) = \sum_{(n,v)=1} \frac{\lambda_{un}\psi(n)}{n^s}, \qquad \text{and} \qquad \overline{A}(s,\psi;u,v) =\sum_{(n,v) =1} \frac{\overline{\lambda}_{un}\psi(n)}{n^s}, 
 \end{equation}
 we can bound the sum over $H$ and $K$ in \eqref{7.04} by 
 $$ 
 \ll \sum_{k} \frac{1}{k} |A(\tfrac 12+s_1, \psi;g_1g_4k,g_3c) \overline{A}(\tfrac 12+ s_2, \overline{\psi}; g_1 g_3k, g_4c)|.
 $$ 
 Since $\lambda_n =0$ for $n> Q^{\vartheta}$, the above sum over $k$ is just a finite sum.  
 
 Inputting these estimates into \eqref{7.02} we find that the contribution of $a$, $e$, $\ell$ in dyadic blocks of size $A$, $E$, $L$ to \eqref{7.01} is 
 \begin{align} 
 \label{7.06} 
 \ll \frac{Q^{1+\epsilon}}{A^2 E^2L} \sum_{A \le a <2A} &\max_{\substack{u_1, u_2 \\ v_1, v_2}} 
 \sum_{\substack{E\le e < 2E\\ L\le \ell <2L}} \sum_{\substack{\psi \bmod{ae \ell} \\ \psi \neq \psi_0 } } \int_{\substack{ \text{Re}(s_1) = \epsilon \\ 
 \text{Re}(s_2) = \epsilon \\ \text{Re} (z) =\epsilon } } | {\mathcal V}_{\alpha,\beta}(s_1, s_2;z)|  |L(\tfrac 12+s_1 +\alpha+ z,\psi)|\nonumber \\
 &\times | L(\tfrac 12+s_2+\beta +z,\overline{\psi})| 
 |A(\tfrac 12+ s_1, \psi;u_1,v_1) A(\tfrac 12+ s_2, \psi; u_2, v_2)| |ds_1 ds_2 dz|. 
 \end{align} 
 We arrive at this expression by using the max over $u_1$, $u_2$, $v_1$, $v_2$ in order to free up the dependencies on $g_1$, $g_2$, $g_3$, $g_4$ and $c$.  The 
 extra factor of $A$ above arises from the requirement that $g_1g_2g_3g_4$ must be a multiple of $a$.  
 
 Now we split the integrals over $s_1$, $s_2$ into dyadic blocks.   From Lemma 5 we see that $\widetilde{\mathcal V}_{\alpha,\beta}(s_1,s_2;z)$ decays 
 rapidly in $|z|$.  We suppose that $|s_1|$ is of size $S_1$ and that this is larger than $|s_2|$.   Inputting the bounds of Lemma 5, the contribution from this dyadic block to the integrals in \eqref{7.06} is (for any natural numbers $k_1$ and $k_2$)  
 \begin{align*}
 \ll Q^{\epsilon} \frac{(1+Q^{\vartheta}/L)^{k_1-1}}{S_1^{k_1}} \int_{\text{Re}(z) = \epsilon} \frac{1}{|z|^{k_2}} &\int_{\substack{\text{Re}(s_1)=\text{Re}(s_2)=\epsilon\\ 
 |s_1|, |s_2| \le 2S_1}}  |L(\tfrac 12+s_1 +\alpha+ z,\psi) L(\tfrac 12+s_2+\beta +z,\overline{\psi})| \\
 &\times  |A(\tfrac 12+ s_1, \psi;u_1,v_1) A(\tfrac 12+ s_2, \psi; u_2, v_2)| |ds_1 ds_2 dz|. 
 \end{align*}
We insert this in \eqref{7.06}, and group terms $r=ae$ so that $r$ lies between $AE$ and $4AE$.  We are left with estimating 
\begin{align} 
\label{7.07} 
&\frac{Q^{1+\epsilon}(1+Q^{\vartheta}/L)^{k_1-1}}{A^2 E^2 L S_1^{k_1}} \sum_{AE \le r <4AE} 
\max_{\substack{u_1, u_2 \\ v_1, v_2}} \sum_{\substack{ L \le\ell  < 2L \\ \psi\bmod{r\ell} \\ \psi\neq \psi_0}} \int_{\text{Re}(z) = \epsilon} \frac{1}{|z|^{k_2}} \int_{\substack{\text{Re}(s_1)=\text{Re}(s_2)=\epsilon\\ 
 |s_1|, |s_2| \le 2S_1}} |ds_1 ds_2 dz| \nonumber \\ 
 &|L(\tfrac 12+s_1 +\alpha+ z,\psi) L(\tfrac 12+s_2+\beta +z,\overline{\psi})|  |A(\tfrac 12+ s_1, \psi;u_1,v_1) A(\tfrac 12+ s_2, \psi; u_2, v_2)|. 
 \end{align}

At this stage we appeal to the following consequence of the hybrid large sieve inequality. 

\begin{proposition} \label{prop1} Given a character $\chi$, put 
$$
M(s,\chi)=\sum_{n\le N} a(n) \chi(n)n^{-s}
$$ 
where $a(n) \ll n^{\epsilon}$.  
Let $\sigma \ge 1/2$ be a real number, and $z$ a complex number with Re$(z) >0$.  Let $k$ be a positive integer, and $Q \ge 2$ and $T\ge 2$ be real 
numbers.  Then 
\begin{align*}
\sum_{q\le Q} \sum_{\substack{ \chi \bmod {qk} \\ \chi \neq\chi_{0}}}  \Big(\int_{-T}^T &|L(\sigma+it + z, \chi) M(\sigma+ it,\chi)| dt \Big)^2 \\
& \ll (kQTN(1+|z|))^{\epsilon} 
\Big( kQ^2 T^2 (1+|z|)^{\frac 12} + QNT (1+|z|)^{\frac 12} k^{\frac 12}\Big) . 
\end{align*} 
\end{proposition} 

Postponing the proof of the proposition for the moment, we first finish the treatment of ${\mathcal U}^{E}_{\alpha,\beta}$.   From Proposition \ref{prop1} 
we may bound the quantity in \eqref{7.07} by 
\begin{align*} 
\ll \frac{Q^{1+\epsilon}(1+Q^{\vartheta}/L)^{k_1-1}}{A^2 E^2 L S_1^{k_1}} \sum_{AE \le r <4AE}  \int_{\text{Re}(z) = \epsilon} \frac{1}{|z|^{k_2}} &(AELQS_1(1+|z|))^{\epsilon} (1+|z|)^{\frac 12}\\
&\times \Big( AE L^2 S_1^2 + A^{\frac 12} E^{\frac 12} L Q^{\vartheta} S_1  \Big) |dz|.
\end{align*} 
Taking $k_2=2$, this simplifies to 
$$ 
\ll \frac{Q^{1+\epsilon}(1+Q^{\vartheta}/L)^{k_1-1}}{A^2 E^2 L S_1^{k_1}} (AELQS_1)^{\epsilon} \Big( A^2 E^2 L^2 S_1^2 + A^{\frac 32} E^{\frac 32} L Q^{\vartheta} S_1 \Big). 
$$ 
Take above $k_1=1$ if $S_1 \le 1+ Q^{\vartheta}/L$ and take $k_1=3$ if $S_1$ is larger.  Then, after summing over all the possible dyadic sizes of $S_1$, the above is seen to be $\ll (AELQ)^{\epsilon}Q(L+ Q^{\vartheta})$.   
Since we may restrict attention to $L \le AL\ll CQ^{\vartheta+\epsilon}$ and $AEL \ll Q^{12}$, it follows that ${\mathcal U}_{\alpha,\beta}^E \ll CQ^{1+\vartheta+\epsilon}$ as needed.

 \begin{proof}[Proof of Proposition \ref{prop1}]  We may clearly assume that the coefficients $a(n)$ are non-zero only when $(n,k)=1$.   
 We first pass from all non-principal characters $\bmod{qk}$ to primitive characters.  Suppose $\chi \bmod{qk}$ is induced by a primitive 
 character ${\widetilde \chi} \bmod{\widetilde{q}\widetilde{k}}$ where $(k, {\widetilde q})= 1$ and ${\widetilde k}$ is composed only of primes dividing $k$.  
 We write $q = d{\widetilde q} \widetilde{k}/(k,{\widetilde k})$ for some integer $d$.   Then we can recast our sum as 
 \begin{equation} 
 \label{7.11} 
 \ll (Qk)^{\epsilon} \sum_{\substack{\widetilde k \\ p|{\widetilde k} \implies p|k}} \sum_{d\le Q(k,\widetilde{k})/\widetilde{k}} \sum_{\substack{\widetilde{q}\le Q(k,\widetilde{k})/({\widetilde k}d) \\ (k,\widetilde{q})=1\\ \widetilde{k}\widetilde{q} >1 }} 
 \sum_{\widetilde{\chi} \bmod{ \widetilde{q}\widetilde{k}} }^* \Big( \int_{-T}^{T} |L(\sigma+it+z,\widetilde{\chi}) M_d(\sigma+it,\widetilde{\chi})|dt \Big)^2. 
 \end{equation} 
 Here the factor $(Qk)^{\epsilon}$ accounts for the difference between $L(s,\chi)$ and $L(s,\widetilde{\chi})$, and $M_d$ indicates that the Dirichlet polynomial $M$ 
 is restricted to terms coprime to $d$.  
 
 Given ${\widetilde k}$ and $d$, we now bound the sum over ${\widetilde q}$ and ${\widetilde \chi}$ using the hybrid large sieve (see Theorem 9.12 of \cite{IK}).   
 We do this in two different ways, depending on whether $T > Q(k,\widetilde{k})/d$ or not.  In the first case, by Cauchy-Schwarz 
 $$ 
  \Big( \int_{-T}^{T} |L(\sigma+it+z,\widetilde{\chi}) M_d(\sigma+it,\widetilde{\chi})|dt \Big)^2 \le \Big( \int_{-T}^{T} |L(\sigma+it+z,\widetilde{\chi})|^2 dt \Big) 
  \Big( \int_{-T}^{T} |M_d(\sigma+it,\widetilde{\chi})|^2 dt\Big). 
  $$ 
 Since $\widetilde{\chi}$ is not principal, approximating the $L$-function by a Dirichlet polynomial of length $\ll \sqrt{(1+|z|)T{\widetilde{q}}{\widetilde{k}}} \ll \sqrt{1+|z|} T$ we may bound the first integral on the 
 right side by $\ll (Qk)^{\epsilon} (1+|z|))^{\frac 12+ \epsilon}T^{1+\epsilon}$.    Therefore we obtain the bound 
 $$ 
 \ll (Qk)^{\epsilon} (1+|z|)^{\frac 12+ \epsilon} T^{1+\epsilon} \sum_{\substack{\widetilde{q}\le Q(k,\widetilde{k})/({\widetilde k}d) \\ (k,\widetilde{q})=1}} 
 \sum_{\widetilde{\chi} \bmod{ \widetilde{q}\widetilde{k}} }^* \int_{-T}^{T} |  M_d(\sigma+it,\widetilde{\chi})|^2dt,  
 $$
 which by the hybrid large sieve is 
 $$ 
 \ll (Qk N)^{\epsilon} (1+ |z|)^{\frac 12+ \epsilon} T^{1+\epsilon} \Big( Q^2 T \frac{(k,\widetilde k)^2}{\widetilde{k}d^2} + N \Big). 
 $$ 
 Summing this over all the possible $\widetilde k$ and $1\le d \le Q(k,{\widetilde k})/{\widetilde k}$ we arrive at the estimate 
 \begin{align} 
 \label{7.12} 
 &\ll (QkNT(1+|z|))^{\epsilon} \sum_{\substack{ \widetilde{k} \\ p| {\widetilde k} \implies p|k}} 
  \Big( Q^2 T^2 (1+|z|)^{\frac 12} \frac{(k,\widetilde{k})^2}{{\widetilde k}} + QNT(1+|z|)^{\frac 12} \frac{(k,\widetilde{k})}{\widetilde k}\Big) \nonumber \\ 
  &\ll (QkNT(1+|z|))^{\epsilon} \Big( k Q^2 T^2 (1+|z|)^{\frac 12} + QNT (1+|z|)^{\frac 12}\Big). 
 \end{align}

 Now we consider the case where $T\le Q (k,\widetilde{k})/d$, where by Cauchy-Schwarz we must estimate 
 $$ 
 T \sum_{\substack{\widetilde{q}\le Q(k,\widetilde{k})/({\widetilde k}d) \\ (k,\widetilde{q})=1}} 
 \sum_{\widetilde{\chi} \bmod{ \widetilde{q}\widetilde{k}} }^* \int_{-T}^{T}  |L(\tfrac 12+ it +z, \widetilde{\chi}) M_d(\sigma+it,\widetilde{\chi})|^2dt.
 $$ 
 An application of the hybrid large sieve bounds the above by 
 $$ 
 \ll (QkNT(1+|z|))^{\epsilon} T \Big( Q^2 T \frac{(k,\widetilde{k})^2}{\widetilde{k} d^2} + N \Big( \frac{QT (k,\widetilde{k}) (1+|z|)}{d}\Big)^{\frac 12}\Big), 
 $$ 
 and summing this over ${\widetilde k}$ and $d\le Q (k,\widetilde{k})/\max(T,\widetilde{k})$ we obtain the estimate 
  \begin{align} 
 \label{7.13} 
 &\ll (QkNT(1+|z|))^{\epsilon} \sum_{\substack{ \widetilde{k} \\ p| {\widetilde k} \implies p|k}} 
 \Big( Q^2 T^2 \frac{(k, \widetilde{k})^2}{\widetilde{k}}  + QN T \Big( \frac{ T(1+|z|) (k, \widetilde{k})}{\max(T, {\widetilde{k}})}\Big)^{\frac 12} \Big)\nonumber \\ 
   &\ll  (QkNT(1+|z|))^{\epsilon}  \Big( Q^2 T^2 k + QNT k^{\frac 12} (1+|z|)^{\frac 12} \Big).   
 \end{align}
 
 Since either \eqref{7.12} or \eqref{7.13} applies, the proposition follows.  
 \end{proof} 
 

\section{The principal character contribution}

\noindent In this section we work on the principal  character contribution arising from \eqref{6.5}.  For technical reasons that will become 
clear later, we introduce a small smoothing of the weight functions appearing in \eqref{6.5}.   Let $\delta$ denote a small parameter, which we 
shall choose as $\delta = Q^{-10}$, and define for positive real numbers $x$, $y$, and $t$
\begin{equation} \label{8.01} 
{\mathcal W}_{\alpha,\beta}(x,y;t) = \frac{1}{2\delta} \int_{-\delta}^{\delta} W_{\alpha,\beta}(|x \pm e^{\xi} y|) V_{\alpha,\beta}\Big( \frac{t}{|x\pm e^{\xi} y|} \Big) d\xi. 
\end{equation} 
A small calculation, using the rapid decay of $V_{\alpha,\beta}$ and that $W_{\alpha,\beta}$ is supported on $[1,2]$, shows that (for some constant $\tau >0$) 
\begin{equation} 
\label{8.02} 
{\mathcal W}_{\alpha,\beta} (x,y;t) = W_{\alpha,\beta}(|x\pm y|) V_{\alpha,\beta}\Big( \frac{t}{|x\pm y|}\Big) + O(\delta y e^{-\tau t}). 
\end{equation}  
Therefore we may recast \eqref{6.5} as 
\begin{align} 
\label{8.03} 
{\mathcal U}_{\alpha,\beta}^0(h,k) & = \frac{Q}{2} \Big(\frac{Q}{\pi}\Big)^{\frac{\alpha+\beta}{2}} \sum_{\substack{ (c,hk)=1 \\ c\le C}} \frac{\mu(c)}{c} 
 \sum_{\substack{ m, n\\ (c,mn)=1\\ mh\neq nk}} \sum_{ \substack{a | g \\ (e,g)= 1} } \mu (a) \frac{\mu(e)}{e}\frac{1}{m^{\frac 12+\alpha}} \frac{1}{n^{\frac 12+\beta}}  \nonumber \\ 
 &\hskip 1 in \sum_{\substack{ \ell \\ (ae \ell, mnhk/g^2)=1} } \frac{1}{\phi(ae\ell)} {\mathcal W}_{\alpha, \beta} \Big( \frac{cmh}{Qg\ell}, \frac{cnk}{Qg\ell}; \frac{\pi mn}{Q}\Big)  +O(Q^{-1}).  
 \end{align}

Treating $c$, $m$, $n$, $a$ and $e$ as given, we work on simplifying the sum over $\ell$ in \eqref{8.03}.   By Mellin inversion, 
the sum over $\ell$ may be written as 
\begin{equation} 
\label{8.1}
\sum_{\substack{\ell \\ (ae\ell, mnhk/g^2)=1} } \frac{1}{\phi(ae\ell)} \frac{1}{2\pi i} \int_{(\epsilon)} \ell^{-w} {\widetilde {\mathcal W}}_{\alpha,\beta}(w) dw, 
\end{equation} 
where (suppressing the dependence of ${\widetilde {\mathcal W}}_{\alpha,\beta}$ on $c$, $m$, $n$, $h$, $k$, $g$) 
\begin{equation} 
\label{8.2} 
\widetilde{{\mathcal W}}_{\alpha,\beta}(w) = 
 \int_0^\infty  {\mathcal W}_{\alpha,\beta} \Big( \frac{cmh}{Qgx}, \frac{cnk}{Qgx}; \frac{\pi mn}{Q}\Big)  x^w \frac{dx}{x}. 
 \end{equation} 
 
 \begin{lemma} \label{lem8.1}  Let $s$ be a complex number with Re$(s)>0$, and let $u$ and $v$ be coprime natural numbers. 
 Then 
 $$
 \sum_{\substack{ \ell =1 \\ (\ell,v)=1}}^{\infty} \frac{1}{\phi(u\ell) \ell^s}  = \frac{1}{\phi(u)} \zeta(1+s) R(s;u,v), 
 $$ 
 where 
 $$
 R(s;u,v) = \prod_{p| v} \Big(1- \frac{1}{p^{s+1}} \Big) \prod_{p\nmid uv} \Big(1 + \frac{1}{p^{s+1}(p-1)}\Big) 
 $$ 
 converges absolutely in Re$(s) > -1$. 
 \end{lemma} 
  \begin{proof}   This is a simple verification using Euler products. 
  \end{proof} 
  
 We use Lemma \ref{lem8.1} in \eqref{8.1}  and move the line of integration to the $(-\epsilon)$ line.  Computing the 
 residue of the pole at $w=0$,  the quantity  in \eqref{8.1} becomes 
 \begin{equation} 
 \label{8.3} 
\frac{ {\widetilde {\mathcal W}}_{\alpha,\beta}(0) }{\phi(ae)} R(0;ae,mnhk/g^2) + \frac{1}{2\pi i \phi(ae)} \int_{(-\epsilon)} \zeta(1+w) R(w;ae, mnhk/g^2) 
{\widetilde {\mathcal W}}_{\alpha,\beta}(w) dw.
\end{equation}  
Using this in \eqref{8.03}, the quantity ${\mathcal U}_{\alpha,\beta}^0(h,k)$ splits into ${\mathcal U}_{\alpha,\beta}^1(h,k) + {\mathcal U}_{\alpha,\beta}^2(h,k)+O(Q^{-1})$, 
with the first term denoting the contribution of the residue at $w=0$, and the second term denoting the remaining integral.

 \subsection{The contribution of ${\mathcal U}_{\alpha,\beta}^1(h,k)$}    Here we show that ${\mathcal U}_{\alpha,\beta}^1(h,k)$ 
 cancels out the contribution of ${\mathcal L}^0_{\alpha,\beta}(h,k)$.  First we perform the sum over $a$ and $e$.  Note that 
 ${\widetilde {\mathcal W}}_{\alpha,\beta}$ is independent of $a$ and $e$, and for any $w$ with Re$(w)>-1$ we have (with a 
 small calculation using that $a|g$ and $(e,g)=1$ so that $(a,e)=1$)   
 \begin{align} 
 \label{8.4} 
 \sum_{\substack{a|g \\ (a,mnhk/g^2)=1} } \mu(a) \sum_{(e,mnhk/g)=1} &\frac{\mu(e)}{e\phi(ae)} R(w;ae, mnhk/g^2)  
 =  
 \prod_{p|mnhk/g^2} \Big( 1-\frac{1}{p^{1+w}}\Big) \nonumber \\
 & \times \prod_{\substack{ p|g \\ p \nmid mnhk/g^2}} \Big( 1+ \frac{1}{p^{1+w}(p-1)} -\frac{1}{p-1}\Big) \prod_{p\nmid mnhk/g} \Big( 1+ \frac{p^{-w}-1}{p(p-1)} \Big). 
 \end{align} 
 Evaluating this at $w=0$, we see that  the sum of the first term in \eqref{8.3} over the 
 relevant $a$ and $e$ is 
 $$ 
 {\widetilde {\mathcal W}}_{\alpha,\beta}(0) \prod_{p| mnhk/g^2} \Big(1-\frac 1p\Big) \prod_{\substack{p|g \\ p \nmid{mnhk/g^2}}} \Big(1-\frac 1p\Big) =  {\widetilde {\mathcal W}}_{\alpha,\beta}(0)\frac{\phi(mnhk)}{mnhk}. 
 $$  
 
 Now from \eqref{8.2} and a change of variables we find 
 $$ 
 {\widetilde {\mathcal W}}_{\alpha,\beta}(0) = 2 \int_0^{\infty} W_{\alpha,\beta}\Big(\frac{x}{Q}\Big) V_{\alpha,\beta} \Big( \frac{\pi mn}{x}\Big) \frac{dx}{x}, 
 $$ 
 where the factor $2$ arises from our convention that $| mh \pm nk|$ indicates a sum over both possible signs. 
 Combining this with our calculations above, we conclude that 
 \begin{align*}
 {\mathcal U}_{\alpha,\beta}^{1}(h,k) = {Q}\Big(\frac{Q}{\pi} \Big)^{\frac{\alpha+\beta}{2}} \sum_{\substack{(c,hk)=1 \\ c\le C} } 
 \frac{\mu(c)}{c} \sum_{\substack{ m, n \\ (c,mn)=1 \\ mh\neq nk}}& \frac{1}{m^{\frac 12+\alpha} n^{\frac 12+\beta}} \frac{\phi(mnhk)}{mnhk} 
\\
& \Big( \int_0^{\infty} W_{\alpha,\beta}\Big(\frac{x}{Q}\Big) V_{\alpha,\beta} \Big( \frac{\pi mn}{x}\Big) \frac{dx}{x}\Big).
\end{align*} 
 Note that this cancels (up to an acceptable error) with our expression for ${\mathcal L}^0_{\alpha,\beta}(h,k)$ in \eqref{5.3}. 

  \subsection{The terms $\mathcal{U}^2_{\alpha,\beta}(h,k)$}  Write 
  \begin{equation} 
  \label{8.5} 
  R_1(w; u,v) = \prod_{p|v} \Big(1-\frac{1}{p^{1+w}}\Big) \prod_{\substack{p|u \\ p\nmid v}} \Big( 1+ 
  \frac{1}{p^{1+w}(p-1)} - \frac{1}{p-1}\Big) \prod_{p \nmid uv} \Big( 1+ \frac{p^{-w}-1}{p(p-1)}\Big), 
  \end{equation} 
  so that the right hand side of \eqref{8.4} equals $R_1(w;g,mnhk/g^2)$.  
  Using the notation \eqref{8.5}, together with \eqref{8.4}, \eqref{8.3}, \eqref{8.1} and finally \eqref{8.03}, we see that 
  \begin{align} 
  \label{8.6} 
  {\mathcal U}^2_{\alpha,\beta}(h,k) &= \frac Q2 \Big(\frac{Q}{\pi}\Big)^{\frac{\alpha+\beta}{2}} \sum_{\substack{c\le C \\ (c,hk)=1}}
  \frac{\mu(c)}{c} \sum_{\substack{m , n \\ (c,mn)=1 \\ mh\neq nk }} \frac{1}{m^{\frac 12+\alpha} n^{\frac 12+\beta}} \nonumber \\ 
  &\hskip 1 in \times \frac{1}{2\pi i} \int_{(-\epsilon)} 
  \zeta(1+w) R_1(w;g,mnhk/g^2) \widetilde{\mathcal W}_{\alpha,\beta}(w) dw. 
  \end{align}
  Recall that $\widetilde{\mathcal W}_{\alpha,\beta}$ was defined in \eqref{8.2}, and that it depends on $c$, $m$, $n$, $h$, $k$, and $g$.  
  
  We now express $\widetilde{\mathcal W}_{\alpha,\beta}(w)$ in a more useful form, making explicit the dependencies on $c$, $m$, $n$, $h$, $k$, $g$, before continuing with the evaluation of \eqref{8.6}.   With the change of variable $y=Qgx/(c|mh \pm e^{\xi} nk|)$ (the denominator can be zero for at most one value of $\xi$, and this is 
  irrelevant for the integral over $\xi$ below) we obtain 
  \begin{equation} 
  \label{8.61} 
  {\widetilde{\mathcal W}}_{\alpha,\beta}(w) =\Big( \frac{1}{2\delta} \int_{-\delta}^{\delta}  \Big( \frac{c|mh\pm e^{\xi} nk|}{Qg}\Big)^w d\xi \Big) 
  \Big( \int_0^{\infty} W_{\alpha,\beta}(1/y) V_{\alpha,\beta}\Big( 
  \frac{\pi mn}{Q}y \Big) y^w \frac {dy}{y}\Big). 
  \end{equation} 
  Using \eqref{3.2}, we see that the second factor above is 
  \begin{align} 
  \label{8.7} 
    &=  \frac{1}{2\pi i} \int_{(1)} 
    {\widetilde V}_{\alpha,\beta}(s) \Big( \frac{Q}{\pi mn}\Big)^s \Big( \int_0^{\infty} W_{\alpha,\beta}(1/y) y^{w-s} \frac{dy}{y} \Big) \frac{ds}{s} \nonumber 
    \\ 
    &=  \frac{1}{2\pi i} \int_{(1)} 
    {\widetilde V}_{\alpha,\beta}(s) \Big( \frac{Q}{\pi mn}\Big)^s  {\widetilde W}_{\alpha,\beta}(s-w) \frac{ds}{s}. 
    \end{align} 
   We point out the difference between ${\widetilde {\mathcal W}}_{\alpha,\beta}$ and $\widetilde{W}_{\alpha,\beta}$ (the Mellin transform of $W_{\alpha,\beta}$), which    we hope will not cause confusion.      
   
   To handle the first factor in \eqref{8.61} we require the following integration formula, which is where the smoothing 
   of the weight functions introduced above will be useful.  
  Given complex numbers $z$ and $w$, we 
define 
 \begin{equation}
 \label{7.1} 
 \mathcal{H}(w,z) =2^{z}\sin\Big(\frac{\pi z}{2}\Big) \Gamma(1-z)\frac{\Gamma(\frac w 2)\Gamma(\frac{z-w}{2})}{\Gamma(\frac{1-w}{2})\Gamma(\frac{1-z+w}{2})} = 
 \sqrt{\pi} \frac{\Gamma(\frac{1-z}{2}) \Gamma(\frac{w}{2}) \Gamma(\frac{z-w}{2})}{\Gamma(\frac z2) \Gamma(\frac{1-w}{2}) \Gamma(\frac{1-z+w}{2})}. 
 \end{equation}
The  variables $w$ and $z$ in \eqref{7.1} and the following proposition are temporary, and in particular, we warn that $w$ is not the variable of \eqref{8.61}.  

 
 \begin{proposition}   Let $z=x+iy$ be a complex number with Re$(z)=x >0$.    Then for any $0< c < x$, and $r>0$ with $r\neq 1$, we 
 have 
 \begin{equation} 
 \label{7.2}
 |1\pm r|^{-z} = |1+r|^{-z} + |1-r|^{-z} = \frac{1}{2\pi i} \int_{(c)} {\mathcal H}(w,z) r^{-w}~dw. 
 \end{equation} 
 Therefore, for any $\delta >0$,  
\begin{equation} 
\label{7.3}
 \frac{1}{2\delta} \int_{-\delta}^{\delta} |1 \pm e^{\xi} r|^{-z} d\xi = \frac{1}{2\pi i} \int_{(c)} {\mathcal H}(w,z) r^{-w} \frac{e^{\delta w}-e^{-\delta w}}{2\delta w} dw, 
 \end{equation} 
 and the integral above converges absolutely for $x <1$.  
\end{proposition}
\begin{proof}   We first establish \eqref{7.2}, whose proof is similar to the standard proofs of  
Perron's formula, see \cite{T}.    If $r >1$ then move the line of integration in \eqref{7.2} to the right.  
We encounter simple poles at $w= z+2k$ for all non-negative integers $k$.  Noting that the contour is oriented clockwise, the 
contribution of the residue at $w=z+2k$ equals 
$$ 
\sqrt{\pi} \frac{\Gamma(\frac{1-z}{2}) \Gamma(\frac{z}{2}+k)}{\Gamma(\frac z2) \Gamma(\frac{1-z}{2} -k) \Gamma(k+\frac 12)} \Big( 2 \frac{(-1)^k}{k!} r^{-z -2k}\Big) = 2 r^{-z-2k} \frac{z(z+1)\ldots (z+2k-1)}{(2k)!}.  
$$ 
Summing this over all non-negative integers $k$ gives 
$$ 
2 r^{-z} \sum_{k=0}^{\infty} r^{-2k} \frac{z(z+1)\ldots (z+2k-1)}{(2k)!} = (r+1)^{-z} + (r-1)^{-z}. 
$$ 
If $r<1$ then we move the line of integration to the left, encountering poles at $w=-2k$ for non-negative integers $k$, and argue similarly.  The 
identity \eqref{7.3} follows  by integrating \eqref{7.2}.   

The integral in \eqref{7.2} is only conditionally convergent, and should be interpreted symmetrically as $\lim_{T\to \infty} \int_{c-iT}^{c+iT}$.    The variant  
given in \eqref{7.3} has the advantage of giving absolutely convergent integrals.  To see this, we use Stirling's formula to 
 conclude that ${\mathcal H}(w,z) \ll (1+|w|)^{x-1}$ 
for $w$ bounded away from the poles of ${\mathcal H}(w,z)$.   Therefore if $x<1$, the integral in \eqref{7.3} converges absolutely.   
\end{proof}

We now return to the first factor in \eqref{8.61}, using Proposition 2 to give an expression for it.  
%
  Thus (recalling Re$(w)=-\epsilon$) we have 
  \begin{align*}
  \frac{1}{2\delta} \int_{-\delta}^{\delta}  \Big( \frac{c|mh \pm e^{\xi} nk|}{Qg} \Big)^w  d\xi &= \Big( \frac{cmh}{Qg}\Big)^w \frac{1}{2\delta} \int_{-\delta}^{\delta}  \Big| 1 \pm e^{\xi} \frac{nk}{mh} \Big|^w d\xi  \\ 
  &= 
   \Big( \frac{cmh}{Qg}\Big)^w \frac {1}{2\pi i} \int_{(\epsilon/2)} {\mathcal H}(z,-w) \Big(\frac{mh}{nk}\Big)^z \frac{e^{\delta z} -e^{-\delta z}}{2\delta z} dz. 
  \end{align*} 
 Combining this with \eqref{8.7}, we conclude that 
  \begin{align} 
  \label{8.9}
  {\widetilde {\mathcal W}}_{\alpha,\beta}(w)& = 
  \frac{1}{(2\pi i)^2} \int_{\substack{ \text{Re}(s)=1 \\ \text{Re}(z)=\epsilon/2}} {\widetilde V}_{\alpha,\beta}(s) {\widetilde W}_{\alpha,\beta}(s-w) {\mathcal H}(z,-w) \nonumber \\ 
&\hskip 1  in \times  \Big(\frac{Q}{\pi mn}\Big)^s \Big( \frac{cmh}{gQ}\Big)^w \Big(\frac{hm}{nk}\Big)^z   \frac{e^{\delta z}-e^{-\delta z}}{2\delta z} \frac{ds}{s} dz.
  \end{align}

  We use \eqref{8.9} in the formula \eqref{8.6}, and bring in the sums over $m$ and $n$.    The condition that $mh \neq nk$ can now be 
  discarded with a negligible error of $\ll Q^{1+\epsilon} (hk)^{\epsilon}$ by moving the line of integration in $s$ to Re$(s) =4\epsilon$ here.    Using 
  this to reintroduce the terms $mh=nk$ in \eqref{8.6}, we define 
 (given $z$ with Re$(z) =\epsilon/2$, and 
  $w$ with Re$(w)=-\epsilon$)  
\begin{equation} 
\label{8.10} 
{\mathcal K}(s,z,w; c,h,k) = \sum_{\substack{m, n \\ (mn,c)=1} } \frac{1}{m^{\frac 12+ \alpha}} \frac{1}{n^{\frac 12+\beta}}   \frac{1}{(mn)^s} \frac{m^{z+w}}{g^w n^z} R_1(w;g,mnhk/g^2). 
\end{equation}  
For $z$ and $w$ as above, the sums over $m$ and $n$ are absolutely convergent for Re$(s)> \frac 12+ 4\epsilon$ say.   By considering Euler products, we may write 
\begin{equation} 
\label{8.11} 
{\mathcal K}(s,z,w; c,h,k)= \zeta(\tfrac 12+ \alpha+ s -z-w) \zeta(\tfrac 12+ \beta +s +z)  {\mathcal K}_1(s,z,w; c,h,k), 
\end{equation} 
where ${\mathcal K}_1$ is absolutely convergent for Re$(s) > 4\epsilon$.

  Thus 
  \begin{align} 
  \label{8.12} 
  {\mathcal U}_{\alpha,\beta}^{2} (h,k) &= \frac Q2 \Big(\frac{Q}{\pi} \Big)^{\frac{\alpha+\beta}{2} } \sum_{\substack{ c\le C \\ (c,hk)=1} } 
  \frac{\mu(c)}{c} 
  \frac{1}{(2\pi i)^2} \int_{\substack{ \text{Re}(w) = -\epsilon \\ \text{Re}(z)= \epsilon/2}}  \zeta(1+w) {\mathcal H}(z,-w) \frac{c^w h^{w+z}}{Q^wk^z} 
  \nonumber \\ 
 & \hskip .25 in \times \frac{1}{2\pi i} \int_{\text{Re}(s) =1}  \Big( \frac{Q}{\pi}\Big)^s {\widetilde V}_{\alpha,\beta}(s) \widetilde{W}_{\alpha,\beta}(s-w) {\mathcal K}(s,z,w;c,h,k) 
 \Big( \frac{e^{\delta z} -e^{-\delta z}}{2\delta z} \Big) 
  \frac{ds}{s} dz dw. 
  \end{align} 
Now we use \eqref{8.11}, and move the line of integration in $s$ to Re$(s) = 4\epsilon$.   In doing so, we encounter poles at 
$s= \frac 12 + z +w -\alpha$, and $s= \frac 12 -\beta -z$, and the remaining integrals with Re$(s)=4\epsilon$ may be bounded by $\ll Q^{1+\epsilon} (hk)^{\epsilon}$. 
  
 Next we examine the contribution of the pole at $s= \frac 12 -\beta -z$ (which will turn out to be unimportant).   This contributes to ${\mathcal U}_{\alpha, \beta}^2 (h,k)$ an amount 
 \begin{align*} 
 &\frac Q2 \Big( \frac{Q}{\pi}\Big)^{\frac{\alpha+\beta}{2}} \sum_{\substack{ c\le C \\ (c,hk)=1} } 
  \frac{\mu(c)}{c} 
 \frac{1}{(2\pi i)^2}  \int_{\substack{ \text{Re}(w) = -\epsilon \\ \text{Re}(z)= \epsilon/2}}   \zeta(1+w) {\mathcal H}(z,-w) \frac{c^w h^{w+z}}{Q^wk^z} \zeta(1+\alpha-\beta-2z-w) \\
 &\times {\mathcal K}_1(\tfrac 12-\beta-z,z,w;c,h,k) \Big(\frac{Q}{\pi}\Big)^{\frac 12-\beta-z} \frac{{\widetilde V}_{\alpha,\beta}(\tfrac12 -\beta-z)}{\frac 12-\beta -z} {\widetilde W}_{\alpha,\beta}(\tfrac 12-\beta-z-w) \Big( \frac{e^{\delta z}- e^{-\delta z}}{2\delta z}\Big) dz dw. 
\end{align*}
 Here we move the line of integration in $z$ to Re$(z) =1/2- \epsilon$ and then bound this contribution by $\ll Q^{1+\epsilon} \sqrt{h/k}$, which is acceptable as 
 $$ 
 \sum_{h, k} \frac{|\lambda_h \lambda_k|}{\sqrt{hk}} Q^{1+\epsilon} \sqrt{h/k} \ll Q^{1+\vartheta +\epsilon}.
 $$

  Now we turn to the contribution of the important pole at $s= \frac 12 +z + w-\alpha$, which accounts for new main terms to 
  go with the diagonal contribution ${\mathcal D}_{\alpha, \beta}(h,k)$.   The contribution of this pole to ${\mathcal U}_{\alpha,\beta}^{2}(h,k)$ is 
\begin{align*} 
 \frac Q2 &\Big( \frac{Q}{\pi}\Big)^{\frac{\alpha+\beta}{2}} \sum_{\substack{ c\le C \\ (c,hk)=1} } 
  \frac{\mu(c)}{c} 
 \frac{1}{(2\pi i)^2}  \int_{\substack{ \text{Re}(w) = -\epsilon \\ \text{Re}(z)= \epsilon/2}}   \zeta(1+w) {\mathcal H}(z,-w) \frac{c^w h^{w+z}}{Q^wk^z} \zeta(1+\beta -\alpha +2z +w )\\
 &\times  {\mathcal K}_1(\tfrac 12 -\alpha+z+w, z,w;c,h,k) \Big(\frac{Q}{\pi} \Big)^{\frac 12+z+w-\alpha} \frac{{\widetilde V}_{\alpha,\beta}(\frac 12 -\alpha+z+w)}{\frac 12 -\alpha +z+w} \widetilde{W}_{\alpha,\beta}(\tfrac 12 +z-\alpha) dz dw. 
 \end{align*}
 Now we move the line of integration in $z$ to Re$(z) = -\frac 12+ 3\epsilon$, passing through a pole at $z=(\alpha-\beta-w)/2$.  
The remaining integral contributes $\ll Q^{1+\epsilon}\sqrt{k/h}$ (which as before is acceptable), while the residue of this pole gives 
 \begin{align*} 
 \frac{Q}{2} \Big(\frac{Q}{\pi}\Big)^{\frac{\alpha+\beta}{2}} \sum_{\substack{ c\le C \\ (c,hk)=1} } 
  \frac{\mu(c)}{c}& \frac{1}{2\pi i} \int_{\text{Re}(w)=-\epsilon} \zeta(1+w) {\mathcal H}\Big( \frac{\alpha-\beta-w}{2},-w\Big) \frac{c^wh^{\frac{\alpha-\beta+w}{2}}}{Q^w k^{\frac{\alpha-\beta-w}{2}} }\Big(\frac{Q}{\pi}\Big)^{\frac{1-\alpha-\beta+w}{2}} \\ 
  &\times{\mathcal K}_1\Big(\frac{1-\alpha-\beta+w}{2}, \frac{\alpha-\beta-w}{2},w;c,h,k\Big) 
  \frac{{\widetilde V}_{\alpha,\beta}((1-\alpha-\beta+w)/2)}{(1-\alpha-\beta+w)}\\
  &\times {\widetilde W}_{\alpha,\beta}\Big( \frac{1-\alpha-\beta-w}{2}\Big) \Big( 1+ O(|w| \delta) \Big) dw. 
  \end{align*}
  Since $\delta =Q^{-10}$ the error term above is negligible, and making 
 the substitution $w=-1+\alpha+\beta -2s$ and see that the above equals 
 \begin{align} 
 \label{8.13} 
 -\frac{Q^{2-\frac{\alpha+\beta}{2}}}{2\pi^{\frac{\alpha+\beta}{2}}}& \sum_{\substack{c\le C \\ (c,hk)=1}} \frac{\mu(c)}{c} 
 \frac{1}{2\pi i} \int_{(-\frac 12+\epsilon)} \zeta(\alpha+\beta-2s) {\mathcal H}(\tfrac 12- \beta+s, 1-\alpha-\beta+2s) \nonumber \\
 &\times c^{-1+\alpha+\beta-2s} 
 \Big( \frac{Q\pi }{hk}\Big)^s \frac{h^{\alpha-\frac 12}}{k^{\frac 12-\beta}} {\mathcal K}_1(-s,\tfrac 12-\beta+s,-2s-1+\alpha+\beta;c,h,k)\nonumber \\
 &\times  {\widetilde V}_{\alpha,\beta}(-s) 
 {\widetilde W}_{\alpha,\beta}(1-\alpha-\beta+s) \frac{ds}{s}. 
 \end{align} 
 Summarizing, the expression in \eqref{8.13} evaluates ${\mathcal U}_{\alpha,\beta}^{2}(h,k)$ up to admissible error terms.  
 
 \section{Proof of the Main Theorem} 
 
 \noindent  Putting together our work from the previous sections, we find that the main terms contributing to $\Delta_{\alpha,\beta}(h,k)$ 
 are 
 $$ 
 {\mathcal D}_{\alpha,\beta}(h,k) + {\mathcal D}_{-\beta,-\alpha}(h,k) + {\mathcal U}_{\alpha,\beta}^{2}(h,k) + {\mathcal U}_{-\beta,-\alpha}^2(h,k). 
 $$ 
 Here the quantity ${\mathcal D}_{\alpha,\beta}(h,k)$ is evaluated in \eqref{4.9}, and the main term for ${\mathcal U}_{\alpha,\beta}^2(h,k)$ 
 is given by \eqref{8.13}, and similar expressions hold for ${\mathcal D}_{-\beta,-\alpha}(h,k)$ and ${\mathcal U}_{-\beta,-\alpha}^{2}(h,k)$.  
 We now show that the terms  ${\mathcal D}_{\alpha,\beta}(h,k)$ and ${\mathcal U}_{-\beta,-\alpha}(h,k)$ may be combined together to yield the 
 residue of the pole at $s=0$ of the expression in \eqref{4.9}, which equals 
 \begin{equation} 
 \label{9.1} 
 \Big(\frac{Q}{\pi}\Big)^{\frac{\alpha+\beta}{2}} \frac{(h,k)^{1+\alpha+\beta}}{2h^{\frac 12+\beta} k^{\frac 12+\alpha}} Q^2 \Phi(hk,1) 
 {\widetilde V}_{\alpha,\beta}(0) {\widetilde W}_{\alpha,\beta}(1) \zeta(1+\alpha+\beta) {\mathcal P}(hk;1,1+\alpha+\beta). 
 \end{equation} 
 Similarly the terms ${\mathcal D}_{-\beta,-\alpha}(h,k)$ and ${\mathcal U}_{\alpha,\beta}^2 (h,k)$ may be combined together 
 to obtain the other main term 
 \begin{equation} 
 \label{9.2} 
 \Big(\frac{Q}{\pi}\Big)^{-\frac{\alpha+\beta}{2}} \frac{(h,k)^{1-\alpha-\beta}}{2h^{\frac 12-\alpha} k^{\frac 12-\beta}} Q^2 \Phi(hk,1) 
 {\widetilde V}_{-\beta,-\alpha}(0) {\widetilde W}_{-\beta,-\alpha}(1) \zeta(1-\alpha -\beta) {\mathcal P}(hk;1,1-\alpha -\beta).
 \end{equation}
 This will complete the proof of our theorem (with the main term given in the form \eqref{1.10}).   
 
 We focus on establishing \eqref{9.2}, with the proof of \eqref{9.1} being entirely analogous.  Our goal is to match the integrand in 
 \eqref{8.13} with the corresponding integrand in \eqref{4.9} for ${\mathcal D}_{-\beta,-\alpha}(h,k)$.    After several initial transformations, 
 we shall move the line of integration in \eqref{8.13} from Re$(s)= -\frac 12 +\epsilon$ to Re$(s)=-\epsilon$, which will then permit 
 the sum over $c\le C$ to be extended to infinity.  Then we will obtain the desired perfect matching.  
 
  Note that 
 \begin{equation} 
 \label{9.3} 
 {\widetilde W}_{\alpha,\beta}(1-\alpha-\beta+s) = \int_0^{\infty} W_{\alpha,\beta}(x)x^{1-\alpha-\beta+s}\frac{dx}{x} = 
{\widetilde W}_{-\beta,-\alpha}(1+s). 
 \end{equation} 
 Next note that (using the definitions \eqref{3.1} and \eqref{7.1}) 
 $$ 
 {\mathcal H}(\tfrac 12 -\beta+s,1-\alpha-\beta+2s) {\widetilde V}_{\alpha,\beta}(-s) = {\widetilde V}_{-\beta, -\alpha}(s) \frac{\sqrt{\pi} \Gamma(\tfrac{\alpha+\beta}{2}-s)}{\Gamma(\tfrac{1-\alpha-\beta}{2} +s)}. 
 $$ 
 Using this together with the functional equation for $\zeta(s)$, we obtain 
 \begin{equation} 
 \label{9.4} 
 \pi^{-\frac{\alpha+\beta}{2}+s} \zeta(\alpha+\beta-2s) {\mathcal H}(\tfrac 12 -\beta+s,1-\alpha-\beta+2s) {\widetilde V}_{\alpha,\beta}(-s) 
 = \pi^{\frac{\alpha+\beta}{2}-s} \zeta(1-\alpha-\beta+2s) {\widetilde V}_{-\beta,-\alpha}(s). 
 \end{equation} 

We now turn to the quantity ${\mathcal K}_1(-s,\frac 12-\beta+s,-2s-1+\alpha+\beta;c,h,k)$, which is defined through \eqref{8.10}, \eqref{8.11}, and \eqref{8.5}. 
These definitions are all in terms of Euler products, and therefore we need only compute the relevant Euler factor for ${\mathcal K}_1$ at each prime $p$.  
This is a straightforward computation, and we give a sketch for a generic prime $p \nmid chk$, and then content ourselves with stating the answers in the other 
cases.  Suppose $p\nmid chk$, and let $\mu$ denote the power of $p$ dividing $m$, and $\nu$ the power of $p$ dividing $n$.   Then the corresponding Euler factor 
for ${\mathcal K}_1(-s,\frac 12 -\beta+s, -2s-1+\alpha+\beta; c, h, k)$ equals 
$$
\Big(1 -\frac{1}{p}\Big)^2 \Big\{ S_1 + S_2 + S_3 \Big\}, 
$$ 
where $S_1$ accounts for terms with $\mu=\nu \ge 0$ (so that $p^{\mu}$ is the power of $p$ dividing $g$), $S_2$ for  terms with $\mu > \nu \ge 0$ (so that $p^{\nu}$ 
is the power of $p$ dividing $g$), and $S_3$ for the terms with $\nu > \mu \ge 0$ (so that $p^{\mu}$ is the power of $p$ dividing $g$).   Recalling \eqref{8.5}, we 
see that 
\begin{align*} 
S_1 &= 1 + \frac{p^{1+2s-\alpha-\beta}-1}{p(p-1)} + \sum_{\mu \ge 1} \frac{1}{p^{\mu(1+\alpha+\beta-2s)} } \Big( 1+ \frac{1}{p^{\alpha+\beta-2s}(p-1)} -\frac{1}{p-1}\Big) \\
&= \Big( 1+ \frac{1}{p^{1+\alpha+\beta-2s}}\Big) \Big( 1- \frac{1}{p^{1+\alpha+\beta-2s}}\Big)^{-1} - \frac{1}{p(p-1)}. 
\end{align*}
Similarly, we see that 
$$ 
S_2 = \Big( 1-\frac{1}{p^{\alpha+\beta-2s}}\Big) \sum_{\mu > \nu\ge 0} \frac{1}{p^{\mu}} \frac{1}{p^{\nu(\alpha+\beta-2s)}} = \frac{1}{(p-1)} \Big( 1-\frac{1}{p^{\alpha+\beta-2s}}\Big) 
\Big( 1- \frac{1}{p^{1+\alpha+\beta-2s}} \Big)^{-1}. 
$$ 
Finally $S_3$ evaluates to be exactly the same as $S_2$.   Putting all these together, we find that the Euler factor for a prime $p\nmid chk$ equals 
\begin{equation} 
\label{9.5} 
\Big(1 -\frac{1}{p }\Big)^2 \Big( \frac{p^2+p-1}{p(p-1)} \Big) = \Big( 1 - \frac{2}{p^2} + \frac 1{p^3} \Big). 
\end{equation} 
For a prime $p|c$ (and so necessarily $p\nmid hk$) the corresponding Euler factor is 
\begin{equation} 
\label{9.6} 
\Big(1-\frac 1p\Big)^2 \Big( 1 + \frac{p^{1-\alpha-\beta+2s}-1}{p(p-1)}\Big) = 
\Big(1-\frac 1p\Big) \Big( 1-\frac 1p -\frac{1}{p^2} + \frac{1}{p^{1+\alpha+\beta-2s}}\Big). 
\end{equation} 
Lastly, suppose $p|hk$ and let $\eta$ denote the power of $p$ dividing $h$, and $\kappa$ the power dividing $k$ (with $\eta+\kappa \ge 1$).  
Here the Euler factor simplifies (after a moderate amount of calculation) to give 
\begin{equation} 
\label{9.7} \Big(1-\frac 1p \Big) p^{\min(\kappa,\eta) (1+2s -\alpha-\beta)}. 
\end{equation}

  Combining \eqref{9.3}, \eqref{9.4}, and the Euler factor computations of \eqref{9.5}, \eqref{9.6}, and \eqref{9.7}, we 
  may recast the integral in \eqref{8.13} as 
  \begin{align} 
  \label{9.8} 
  -\frac{Q^2}{2} \Big( \frac{Q}{\pi}\Big)^{-\frac{\alpha+\beta}{2}} \frac{\phi(hk)}{hk} \frac{(h,k)^{1-\alpha-\beta}}{h^{\frac 12-\alpha} k^{\frac 12-\beta}} 
&\frac{1}{2\pi i} \int_{(-\frac 12+\epsilon)}  \zeta(1-\alpha-\beta+2s) {\widetilde V}_{-\beta,-\alpha}(s) {\widetilde W}_{-\beta,-\alpha}(1 +s) 
\nonumber \\
&\times \Big( \frac{Q}{\pi HK}\Big)^s  \sum_{\substack{c\le C \\ (c,hk)=1}} \frac{\mu(c)}{c^{2-\alpha-\beta+2s}}  \prod_{p\nmid chk} \Big( 1-\frac{2}{p^2} 
+ \frac{1}{p^3}\Big) \nonumber \\
&\times \prod_{p|c} \Big( 1-\frac{2}{p} +\frac{1}{p^3} +\frac{1}{p^{1+\alpha+\beta-2s}} -\frac{1}{p^{2+\alpha+\beta-2s}} \Big) \frac{ds}{s}. 
\end{align}
 We move the line of integration to Re$(s)=-\epsilon$, and extend the sum over $c$ to infinity incurring an error of $\ll Q^{2+\epsilon} \frac{(h,k)}{\sqrt{hk}} C^{-1}$.   The 
 extended sum over $c$ in \eqref{9.8} equals (writing $\gamma = 1-\alpha-\beta+2s$ for brevity) 
 \begin{align*} 
& \sum_{(c,hk)=1} \frac{\mu(c)}{c^{1+\gamma} } \prod_{p\nmid chk} \Big( 1- \frac {2}{p^2} + \frac{1}{p^3} \Big) 
 \prod_{p|c} \Big( 1- \frac 2p + \frac{1}{p^3} + \frac{1}{p^{2-\gamma}} - \frac{1}{p^{3-\gamma}}\Big) 
 \\
 = &\prod_{p \nmid hk} \Big( 1- \frac{2}{p^2} + \frac 1{p^3} - \frac{1}{p^{1+\gamma}} \Big( 1- \frac{2}{p} + 
 \frac{1}{p^3} + \frac{1}{p^{2-\gamma}} - \frac{1}{p^{3-\gamma}} \Big) \Big)  = {\mathcal P}(hk;1,\gamma). 
 \end{align*}
 Therefore  \eqref{9.8} becomes (note that $\Phi(hk,1) = \phi(hk)/(hk)$)
 \begin{align*}
   -\frac{Q^2}{2} \Big( \frac{Q}{\pi}\Big)^{-\frac{\alpha+\beta}{2}} \Phi(hk,1) \frac{(h,k)^{1-\alpha-\beta}}{h^{\frac 12-\alpha} k^{\frac 12-\beta}} 
\frac{1}{2\pi i} \int_{(-\epsilon)}  &\zeta(1-\alpha-\beta+2s) {\widetilde V}_{-\beta,-\alpha}(s) {\widetilde W}_{-\beta,-\alpha}(1 +s)
\\
&\times \Big( \frac{Q}{\pi HK}\Big)^s {\mathcal P}(hk;1,1-\alpha-\beta+2s) \frac{ds}{s}. 
\end{align*}
Adding this expression to the analogue of \eqref{4.9} for ${\mathcal D}_{-\beta,-\alpha}(h,k)$, we see that their combined contribution gives 
precisely the residue of at $s=0$, which establishes \eqref{9.2}.     

Thus we have shown that the main terms for $\Delta_{\alpha,\beta}(h,k)$ are exactly of the shape given in the main theorem.  When averaged 
over $h$ and $k$, the various remainder terms that arose are bounded by 
$$ 
\sum_{h, k \le Q^{\vartheta}} \frac{\lambda_h \overline{\lambda_k}} {\sqrt{hk}} {\mathcal E}_{h,k} \ll \frac{Q^{2+\epsilon}}{C} + CQ^{1+\vartheta+\epsilon}. 
$$ 
Choosing finally $C= Q^{(1-\vartheta)/2}$, the proof of the main theorem is complete.

\end{document}